\documentclass{article}
\title{The eta-inverted sphere over the rationals}

\author{Glen Matthew Wilson}

\usepackage{lscape}
\usepackage{setspace}
\usepackage{amssymb,amsfonts,amsmath,amsthm}
\usepackage{cite, url} 
\usepackage{tikz}
\usepackage{xy} 
\usepackage{rotating}
\usepackage{bbm}
\usepackage[pdfencoding=auto]{hyperref}
\xyoption{all} 
\newdir{ >}{{}*!/-12pt/@{>}} 
\newdir{ -}{{}*!/-5pt/@{}} 
\newdir{- }{{}*!/-5pt/@{}} 
\newdir{ ^(}{{}*!/-5pt/@^{(}} 
\newdir{ _(}{{}*!/-5pt/@_{(}}
\newdir{ |}{{}*!/-5pt/@{|}}
\newdir{_(}{{}*!/0pt/@_{(}}
\xyoption{arc}
\xyoption{ps}
\AtBeginDocument{%
  \setlength{\abovedisplayskip}{4pt}%
  \setlength{\belowdisplayskip}{4pt}%
}
\newcommand{\Pihat}{\hat{\Pi}}
\newcommand{\myol}[2][3]{{}\mkern#1mu\overline{\mkern-#1mu#2}}

\newcommand{\onetwomid}[2]{\draw[line width=.6pt] (#1, #2) -- (#1, #2+1);
\foreach \n in {0,1}{\filldraw [black, fill=blue] (#1, \n+#2+.1) circle (2pt);}
\foreach \n in {0}{\filldraw [black, fill=red] (#1-.3, \n+#2+.2) circle (2pt);} 
\foreach \n in {1}{\filldraw [black, fill=green] (#1-.3, \n+#2+.2) circle (2pt);} 
}

\newcommand{\newdtwomed}[2]{\draw [line width=.6pt, red] (#1, #2+.1)-- (#1-1, #2+.1+1);
\draw [line width=.6pt, red] (#1-.3, #2+.2)-- (#1-1-.3, #2+.2+1);
\draw [line width=.6pt, red] (#1, #2+1+.1)-- (#1-1, #2+2+.1);
\draw [line width=.6pt, red] (#1-.3, #2+1+.2)-- (#1-1-.3, #2+2+.2);
}

\newcommand{\vtwomultmid}[2]{\foreach \n in {0,1}{\draw [line width=.6pt, ->, >=stealth] (#1, #2+.1+\n)-- (#1+.5, #2+.1+.165+\n);}
\foreach \n in {0,1}{\draw [line width=.6pt, ->, >=stealth] (#1-.3, #2+.2+\n)-- (#1+.5-.3, #2+.2+.165+\n);}
}

\newcommand{\Fbar}{\myol{F}}

\newcommand{\twocomp}{{}^{{\kern -.7pt}\wedge}_2}
\newcommand{\ellcomp}{{}^{{\kern -.5pt}\wedge}_{\ell}}
\newcommand{\unit}{\mathbbm{1}}
\newcommand{\SH}{\mathcal{SH}}

\DeclareMathOperator{\Ext}{Ext}
 		
\DeclareMathOperator{\Hom}{Hom} 
\DeclareMathOperator{\colim}{colim}

\newcommand{\st}{ \, \vert \, }

\DeclareMathOperator{\InverseLimit}{\underleftarrow{\lim}}

\newcommand{\MASS}{\frakM}
\newcommand{\BSS}{\frakB}

\newcommand{\Hcomp}{{}^{{\kern -.4pt}\wedge}_{{\kern -.5pt}H}}
\DeclareMathOperator{\cd}{cd}
\newcommand{\presup}[1]{{}^{#1}}
\newcommand{\Ei}{{}^i\mathrm{E}}
\newcommand{\Fi}{{}^i\mathrm{F}}
\newcommand{\Gi}{{}^i\mathrm{G}}
\newcommand{\di}{{}^i\mathrm{d}}

\theoremstyle{definition} 
\newtheorem{proposition}[equation]{Proposition}
\newtheorem{definition}[equation]{Definition} 
\newtheorem{theorem}[equation]{Theorem}
 
\newtheorem{corollary}[equation]{Corollary} 
\newtheorem{lemma}[equation]{Lemma} 
\theoremstyle{remark}

\newcommand{\bbC}{\mathbb{C}}

\newcommand{\bbF}{\mathbb{F}}

\newcommand{\bbN}{\mathbb{N}}

\newcommand{\bbQ}{\mathbb{Q}}
\newcommand{\bbR}{\mathbb{R}}

\newcommand{\bbZ}{\mathbb{Z}}

\newcommand{\calA}{\mathcal{A}} 

\newcommand{\calC}{\mathcal{C}}

\newcommand{\frakB}{\mathfrak{B}}

\newcommand{\frakM}{\mathfrak{M}}

\newcommand{\rmE}{\textrm{E}}
\newcommand{\rmF}{\textrm{F}} 
\newcommand{\rmG}{\textrm{G}}

\begin{document}
\maketitle


\begin{abstract}
  We calculate the motivic stable homotopy groups of the two-complete
  sphere spectrum after inverting multiplication by the Hopf map
  $\eta$ over fields of cohomological dimension at most $2$ with
  characteristic different from 2 (this includes the $p$-adic fields
  $\bbQ_p$ and the finite fields $\bbF_q$ of odd characteristic) and
  the field of rational numbers; the ring structure is also
  determined.
\end{abstract}

\section{Introduction}
\label{sec:introduction}

Guillou and Isaksen laid the foundation for calculating
$\pi_{**}(\unit\twocomp)[\eta^{-1}]$, the motivic stable homotopy
groups of the two-complete sphere spectrum after inverting
multiplication by $\eta$, over the complex numbers using the
$h_1$-inverted motivic Adams spectral sequence in
\cite{GI}. They conjectured a pattern of differentials in the
$h_1$-inverted motivic Adams spectral sequence and identified the
$E_{\infty}$ page of the spectral sequence assuming the
conjecture. Shortly after Guillou and Isaksen's paper appeared,
Andrews and Miller \cite{AM} proved Guillou and Isaksen's
conjecture. All together, these results show
$\pi_{**}(\unit\twocomp)[\eta^{-1}] \cong \bbF_2[\eta^{\pm 1}, \mu,
\varepsilon]/(\varepsilon^2)$
\cite[Conjecture 1.3]{GI}.  Guillou and Isaksen then analyzed
the $h_1$-inverted motivic Adams spectral sequence over the real
numbers and gave a complete calculation of the ring
$\pi_{**}(\unit\twocomp)[\eta^{-1}]$ over the base field $\bbR$
\cite{GI-Real}. 

The subject of this paper is the calculation of
$\pi_{**}(\unit\twocomp)[\eta^{-1}]$ over the field of rational
numbers $\bbQ$ and fields $F$ with $\cd_2(F)\leq 2$ and characteristic
different from 2, such as the $p$-adic fields $\bbQ_p$ and finite
fields $\bbF_q$ of odd characteristic. We write
$\pi_{s,w}(\unit\twocomp)(F)$ for the stable homotopy group
$\SH(F)(\Sigma^{s,w}\unit,\unit\twocomp)$ and frequently abbreviate
this to $\pi_{s,w}(\unit\twocomp)$ if the base field $F$ is clear from
context.

Write $\MASS(F)$ for the motivic Adams spectral sequence at the prime
$2$ over the field $F$ at the motivic sphere spectrum. This spectral
sequence has $E_2$ page given by
$\MASS(F)_2^{f,s,w} =
\Ext_{\calA_{**}(F)}^{f,s+f,w}(H_{**}(F), H_{**}(F))$
and conditionally converges 
\begin{equation*}
  \MASS(F)^{f,s,w} \Rightarrow \pi_{s,w}(\unit\Hcomp)(F)
\end{equation*}
where $\unit\Hcomp$ is the $H$-nilpotent completion of the sphere
spectrum defined by Bousfield \cite[\S5]{Bousfield}. Hu, Kriz, and
Ormsby \cite[Theorem 1]{HKO-Convergence} proved that $\unit\Hcomp$ is
weakly equivalent to $\unit\twocomp$ over a field $F$ of
characteristic $0$ with $\cd_2(F[\sqrt{-1}])<\infty$. Wilson and
{\O}stv{\ae}r \cite[Proposition 5.10]{WO} note that the same argument
works over fields over positive characteristic under the assumption
that $\cd_2(F[\sqrt{-1}])<\infty$.

Given the conditionally convergent spectral sequence
$\MASS(F) \Rightarrow \pi_{**}(\unit\twocomp)(F)$ and that
$\eta \in \pi_{1,1}(\unit\twocomp)(F)$ is detected by
$h_1 \in \MASS(F)^{1,1,1}$, one can try to calculate
$\pi_{**}(\unit\twocomp)(F)[\eta^{-1}]$ using the $h_1$-inverted
spectral sequence, defined as the following colimit of spectral
sequences.
\begin{equation*}
  \MASS(F)[h_1^{-1}] = \colim(\MASS(F) 
  \xrightarrow{h_1 \cdot} \MASS(F) 
  \xrightarrow{h_1\cdot} \MASS(F) \cdots )
\end{equation*} 
It is not obvious that $\MASS(F)[h_1^{-1}]$ converges to
$\pi_{**}(\unit\twocomp)(F)[\eta^{-1}]$. Guillou and Isaksen show that
it does converge for the complex numbers $\bbC$ in
\cite[\S6]{GI} and the real numbers $\bbR$ in
\cite[\S5]{GI-Real}. We address convergence for more general fields in
section \ref{sec:conv}.

The Milnor-Witt $t$-stem of $\unit\twocomp$ over $F$ is the group
$\Pihat_t(F) = \bigoplus_{k \in \bbZ} \pi_{k+t,k}(\unit\twocomp)(F)$.
Note that $\Pihat_0(F)$ is a ring and $\Pihat_t(F)$ is a
$\Pihat_0(F)$-module. Our main results will be stated in terms of
Milnor-Witt stems and the Witt group of quadratic forms $W(F)$. In
many cases, the two-complete $\eta$-inverted Milnor-Witt 0-stem can be
can be described in terms of the Witt ring of quadratic forms of the
field.

\begin{proposition}
  \label{prop:0-stem}
  If $F$ is a field for which the Witt group of quadratic forms $W(F)$
  is finitely generated or $W(F)$ has bounded $2$-torsion exponent
  ($F=\bbQ$, for example), then there is an isomorphism
  $\Pihat_0(F)[\eta^{-1}] \cong W(F)\twocomp[\eta^{\pm 1}]$.
\end{proposition}

\begin{proof}
  Morel has shown there is an isomorphism $\pi_{n,n}(\unit)\cong W(F)$
  for $n\geq 1$ in \cite{Morel12}. For $n\geq 1$ the homotopy group
  $\pi_{n,n}(\unit\twocomp)$ fits into the exact sequence
  \begin{equation*}
    0 \to \Ext(\bbZ/2^{\infty}, W(F)) 
    \to \pi_{n,n}(\unit\twocomp) 
    \to \Hom(\bbZ/2^{\infty}, \pi_{n-1,n}(\unit)) 
    \to 0
  \end{equation*}
  by \cite[Equation (2)]{HKO-Convergence}. But as
  $\pi_{n-1,n}\unit = 0$ by Morel's connectivity theorem
  \cite{Morel12}, there is an isomorphism
  $\Ext(\bbZ/2^{\infty}, W(F)) \to \pi_{n,n}(\unit\twocomp)$.

  If $W(F)$ is finitely generated, there is an isomorphism
  $\Ext(\bbZ/2^{\infty}, W(F)) \cong W(F)\twocomp$ by a result of
  Bousfield and Kan \cite[Chapter VI, Section 2.1]{BousfieldKan},
  hence $\pi_{n,n}(\unit\twocomp)\cong W(F)\twocomp$. If $W(F)$ has
  bounded 2-torsion exponent, then ${}_{2^n}W(F)={}_{2^m}W(F)$ for all
  $n$ and $m$ sufficiently large. The Mittag-Leffler condition is
  satisfied for the tower $\{{}_{2^n}W(F) \}$, hence
  $\InverseLimit^{1} {}_{2^n}W(F) \cong
  \InverseLimit^{1}\Hom(\bbZ/2^n, W(F))$
  vanishes. Now the short exact sequence of \cite[Application
  3.5.10]{HBook}
  \begin{equation*}
    0\to \InverseLimit^{1}\Hom(\bbZ/2^n, W(F)) \to \Ext(\bbZ/2^{\infty},W(F)) \to W(F)\twocomp \to 0
  \end{equation*}
  shows there is an isomorphism
  $\Ext(\bbZ/2^{\infty}, W(F)) \cong W(F)\twocomp$, and so
  $\pi_{n,n}(\unit\twocomp)\cong W(F)\twocomp$.

  Finally, there is an isomorphism
  $\Pihat_0(F)[\eta^{-1}] \cong W(F)\twocomp[\eta^{\pm 1}]$ since for any
  class $\alpha \in \Pihat_0(F)$ and $n$ sufficiently large, the class
  $\eta^n \alpha$ is an element of $\pi_{k,k}\unit\cong W(F)$ with
  $k\geq 1$.
\end{proof}

For finite fields $\bbF_q$, the Milnor-Witt 0 stem is now
determined by the calculation of the Witt group of finite fields, a
standard reference being Scharlau \cite[Chapter 2, 3.3]{Scharlau}.
\begin{equation*}
  \Pihat_0(\bbF_q)[\eta^{-1}] \cong W(\bbF_q)\twocomp[\eta^{\pm 1}] = 
  \begin{cases}
    \bbZ/2[\eta^{\pm 1}, u]/u^2 & q \equiv 1 \bmod 4 \\
    \bbZ/4[\eta^{\pm 1}] & q \equiv 3 \bmod 4 
  \end{cases}
\end{equation*}

We find in theorem \ref{thm:cd2} that for a field $F$ with
$\cd_2(F)\leq 2$ and characteristic different from 2 the two-complete
$\eta$-inverted Milnor-Witt stems take the following form.
\begin{equation*}
  \Pihat_t(F)[\eta^{-1}] \cong
  \begin{cases}
    W(F)\twocomp[\eta^{\pm 1}] & t\geq 0 \text{ and } t\equiv 3 \bmod 4 \text{ or } t \equiv 0 \bmod 4 \\
    0 & \text{otherwise}
  \end{cases}
\end{equation*}
This gives a complete calculation of $\Pihat_*(\bbF_q)[\eta^{-1}]$ for
the finite fields $\bbF_q$ of odd characteristic.

Theorem \ref{thm:Q} calculates the ring $\Pihat_*(\bbQ)[\eta^{-1}]$. In
particular, the Milnor-Witt stems are
\begin{equation*}
  \Pihat_t(\bbQ)[\eta^{-1}] \cong 
  \begin{cases}
    W(\bbQ)\twocomp[\eta^{\pm 1}] & t=0 \\
    W(\bbQ)\twocomp[\eta^{\pm 1}]/2^{n+1}
    & t \geq 0, t \equiv 3 \bmod 4, n = \nu_2(t+1) \\
    M & t \equiv 0 \bmod 4, t \geq 4 \\
    0 & \text{otherwise}
  \end{cases}
\end{equation*}
where $\nu_2(t+1)$ is the 2-adic valuation of $t+1$ and $M$ is the
submodule of $W(\bbQ)\twocomp[\eta^{\pm 1}]$ defined in definition
\ref{def:M}.

The method of proof for the calculations over $\bbQ$ of theorem
\ref{thm:Q} follows the strategy employed by Ormsby and {\O}stv{\ae}r
in \cite{MotBPInvQ} to calculate the homotopy groups of
$BP\langle n\rangle$ over $\bbQ$. First, for each completion
$\bbQ_{\nu}$ of $\bbQ$ one uses the $\rho$-Bockstein spectral sequence
to calculate $\Ext(\bbQ_{\nu})[h_1^{-1}]$ and then the motivic Adams
spectral sequence to calculate
$\pi_{**}(\unit\twocomp)(\bbQ_{\nu})[\eta^{-1}]$. We next follow the
motivic Hasse principle to identify the differentials in the
$\rho$-Bockstein spectral sequence and the motivic Adams spectral
sequence over $\bbQ$ by comparing these spectral sequences with the
associated spectral sequences over the completions.

The result of Ormsby, R{\"o}ndigs, and {\O}stv{\ae}r,\cite[Proof of
Theorem 1.5]{Vanishing} shows that the vanishing
$\Pihat_t(\bbQ)[\eta^{-1}] = 0$ when $t \equiv 1, 2 \bmod 4$ occurs
systematically for all formally real fields $F$ with
$\cd_2(F[i])<\infty$.  Their calculation is used in this paper to show
that $\Pihat_1(\bbQ)[\eta^{-1}]$ vanishes, as it is unclear whether or
not the motivic Adams spectral sequence over $\bbQ$ converges strongly
in Milnor-Witt stem 1.

Ananyevskiy, Levine, and Panin investigate the $\eta$-inverted sphere
spectrum $\unit[\eta^{-1}]$ in \cite{ALP} over fields $F$ of
characteristic different from 2. They find that the stable homotopy
sheaf $\oplus_{n\in\bbZ}\pi_{n,n}\unit[\eta^{-1}]$ is isomorphic to
the sheaf $\underline{W}[\eta^{\pm 1}]$ where $\underline{W}$ is the
Nisnevich sheaf associated to the presheaf of Witt groups (the Witt
group $W(X)$ of an algebraic variety $X$ is defined by Knebusch in
\cite[Chapter I \S5]{Knebusch}). The consequence of this for
calculating stable homotopy groups is that
\begin{equation*}
  \oplus_{n\in\bbZ}\pi_{n,n}(\unit[\eta^{-1}])(F) \cong W(F)[\eta^{\pm
    1}]
\end{equation*}
for all fields $F$ of characteristic different from 2. In addition to
this absolute statement about the $\eta$-inverted Milnor-Witt 0-stem,
they identify the rationalization of $\unit[\eta^{-1}]$ with an object
in the heart of the homotopy $t$-structure on $\SH(F)$ \cite[Theorem
3.4]{ALP} and find that the sheaf $\pi_{s,w}(\unit[\eta^{-1}]_{\bbQ})$
takes the following form.
\begin{equation*}
  \pi_{s,w}(\unit[\eta^{-1}]_{\bbQ}) = 
  \begin{cases}
    \underline{W}_{\bbQ} & \text{if } s=w \\
    0 & \text{otherwise}
  \end{cases}
\end{equation*}
The calculations in this paper are about the $\eta$-inverted
2-complete sphere spectrum $\unit\twocomp[\eta^{-1}]$ in contrast to
Ananyevskiy, Levine, and Panin's results about $\unit[\eta^{-1}]$ and
$\unit[\eta^{-1}]_{\bbQ}$.

We will follow the grading conventions for
$\Ext(F)=\Ext_{\calA_{**}(F)}(H_{**}(F),H_{**}(F))$ employed by
Guillou and Isaksen in \cite[\S2.1]{GI}. In particular, for a
class $x \in \Ext(F)$ in Adams filtration $f$, stem $s$, and weight
$w$, the Milnor-Witt stem of $x$ is $t = s-w$ and the Chow weight of
$x$ is $c = s +f - 2w$. We will write degrees as $\deg(x) = (f,t,c)$
unless otherwise specified. We will frequently use the isomorphism
$\Ext(\bbC)[h_1^{-1}]\cong \bbF_2[h_1^{\pm 1}, v_1^4, v_2,
v_3,\ldots]$
with $\deg(v_1^4) = (4,4,4)$ and $\deg(v_n) = (1, 2^n-1, 1)$
established in \cite[Theorem 3.4]{GI}, and we adopt the
convention of writing $P$ for $v_1^4$.

\section*{Acknowledgments}  

I would like to thank Dan Isaksen for teaching me how to use Massey
product shuffles, Elden Elmanto for the motivation to write up the
convergence result of section \ref{sec:conv}, Jonas Kylling and
H{\aa}kon Kolderup for their comments on an earlier version of this
work, Oliver R{\"o}ndigs for helpful suggestions, and Paul Arne
{\O}stv{\ae}r for his support and encouragement that made this paper
possible. This paper was partially completed while participating in
the research program ``Algebro-geometric and homotopical methods'' at
the Institut Mittag-Leffler. I gratefully acknowledge support from the
RCN project ``Motivic Hopf Equations,'' no.~250399.

\section{Convergence of the \texorpdfstring{$h_1$}{h1}-inverted motivic Adams spectral sequence}
\label{sec:conv}

We refer the reader to Boardman's notes on spectral sequences
\cite{Boardman} for the terminology concerning the convergence of
spectral sequences. The notion of a map of filtered groups compatible
with a map of spectral sequences is defined by Weibel in
\cite[p. 126]{HBook}.

Consider a collection of cohomologically graded spectral sequences
$(\Ei_r, \di_r)$ for $i\in \bbN$ where $\Ei_r^s = 0 $ for all $s<0$
and all $r$, and each spectral sequence $\Ei$ converges strongly to an
abelian group $\Gi$ filtered by $\Fi_s$. Here $r$ indicates the page
of the spectral sequence and $s$ indicates the internal degree. We
will omit these subscripts and superscripts where it is convenient.
These assumptions on our spectral sequences precisely mean the
following conditions hold.
\begin{enumerate}
    \item The filtration is exhaustive with $\Fi_0 = \Gi$,
  \begin{equation*}
    \Gi = \Fi_0 \supseteq \Fi_1 \supseteq \Fi_2 \supseteq \cdots
  \end{equation*}
  
    \item For all $s\in\bbZ$ there are isomorphisms
  $\Ei_{\infty}^s \cong \Fi_s/\Fi_{s+1}$.

    \item The filtration is Hausdorff $\cap_j \Fi_j = 0$.

    \item The filtration is complete
  $\InverseLimit_j \Gi/\Fi_j \cong \Gi$.
\end{enumerate}
We assume the spectral sequences are $\bbZ^k$-graded in addition to
the internal grading $s$. In practice, $k$ is either $1$ or $2$.

Let $\presup{i}f : \Ei \to \presup{i+1}\rmE$ be a directed system of
maps of spectral sequences compatible with
$\presup{i}g : \presup{i}\rmG \to \presup{i+1}\rmG$ of degree
$a\in\bbZ^k$. The colimit of the directed system
$\presup{i}f : \Ei \to \presup{i+1}\rmE$ is again a spectral sequence
$\widetilde{E} = \colim \Ei$ that converges weakly to
$\widetilde{G} = \colim \Gi$ filtered by
$\widetilde{F}_j = \colim \Fi_j$. Under what circumstances can we
guarantee $\widetilde{E}$ converges strongly to $\widetilde{G}$?

The filtration $\widetilde{F}_j$ of $\widetilde{G}$ may fail to be
Hausdorff or complete. To see how the Hausdorff condition may fail,
consider $\Ei = \bbZ$, $\Gi = \bbZ$, and $\Fi_j = 0 $ for $j\geq i$
and $\Fi_j = \Gi$ for $j<i$. If we define maps $\presup{i}f = 0$ and
$\presup{i}g(1) = 1$, then $\widetilde{E} = 0$,
$\widetilde{G} = \bbZ$, and $\widetilde{F}_j = \bbZ$ for all
$j\geq 0$. The issue is that the class $1 \in \Gi$ is in filtration
$i$, which shows $1 \in \widetilde{G}$ is in filtration $i$ for all
natural numbers $i$.

Completeness can fail if each $\Gi$ has a finite filtration but the
colimit $\widetilde{G}$ has an infinite filtration. Consider
$\Gi = \bbZ$, $\Fi_j = 2^j\bbZ$ for $j \leq i$ and $\Fi_j = 0$ for
$j > i$ with the map $\presup{i}g(1) = 1$. Then
$\widetilde{G} = \bbZ$, yet
$\widetilde{E}^s \cong \widetilde{F}_s/\widetilde{F}_{s+1} \cong
\bbZ/2$
for all $s \geq 0$ and $\InverseLimit_s \widetilde{G}/\widetilde{F}_s$
is isomorphic to the $2$-adic integers $\bbZ_2$.

\begin{definition}
  Consider a directed system of $\bbZ^k$-graded spectral sequences
  $\presup{i}f : \presup{i}\rmE \to \presup{i+1}\rmE$ of degree $a$,
  that is, for all $i\in \bbN$ and $x\in\Ei$ the degree of
  $\presup{i}f(x)$ is $a + \deg(x)$, and $\presup{i}f$ does not change
  the internal degree $s$. The directed system $\presup{i}f$ has a
  horizontal vanishing line of height $N$ in the direction $a$ if for
  any degree $b$ there exists $K\in\bbN$ for which the groups
  $\Ei^{s,b+ia}$ vanish for all $i > K$ and $s > N$.
\end{definition}

The term horizontal vanishing line comes from the special case where
for all $i$ we have $\presup{i}E=E$ and $E= \oplus_{s,p}E^{s,p}$ is a
$\bbZ$-graded spectral sequence in $p$ with internal degree $s$.
If one makes a chart for $\presup{i}E$ where the vertical axis is the
internal degree $s$ and the horizontal axis is $p$, a horizontal
vanishing line of height $N$ in the direction $1$ says that $E^{s,p}$
vanishes when $(s,p)$ is above the horizontal line $s=N$ and $p$ is
sufficiently large.

\begin{proposition}
  \label{prop:ss_convergence}
  Suppose $\presup{i}f : \Ei \to \presup{i+1}\rmE$ is a directed
  system of $\bbZ^k$-graded spectral sequences of degree $a$ with a
  horizontal vanishing line of height $N$ in the direction $a$. The
  colimit spectral sequence $\widetilde{E}$ then converges strongly to
  $\widetilde{G}$ with respect to the filtration $\widetilde{F}$.
\end{proposition}

\begin{proof}
  We first show the filtration $\widetilde{F}$ of $\widetilde{G}$ is
  complete. For $b \in \bbZ^k$, the degree $b$ component of
  $\widetilde{G}$ is
  \begin{equation*}
    \widetilde{G}^{b} = \colim( \presup{0}G^{b} \to \presup{1}G^{b+a} \to \cdots \to \Gi^{b+ia}\to \cdots).
  \end{equation*}
  The assumption that there is a horizontal vanishing line of height
  $N$ in the direction $a$ implies for all $i> K$ the filtration of
  $\presup{i}\rmG^{b+ia}$ is finite. This is because the filtration of
  $\presup{i}\rmG$ is Hausdorff and $\presup{i}E^{s, b+ia}$ vanishes
  for $s>N$, so $\presup{i}\rmF_j$ is trivial for $i >K$ and $j >N$.
  Since finite limits and directed colimits commute, it follows that
  \begin{align*}
    \InverseLimit_j(\widetilde{G}^{b}
    / \widetilde{F}_j^{b} )
    & \cong \InverseLimit_j(\colim_{i>K} \presup{i}\rmG/
      \presup{i}\rmF_j )\\
    & \cong \colim_{i>K} (\InverseLimit_j  \presup{i}\rmG/
      \presup{i}\rmF_j)\\
    & \cong \colim_{i>K} \presup{i}\rmG^{b+ia} \\
    & \cong \widetilde{G}^b.
  \end{align*}

  We now show the filtration $\widetilde{F}_j$ of $\widetilde{G}$ is
  Hausdorff. Let $x \in \widetilde{G}^b$ be a nonzero element. Lemma
  \ref{lem:existence} shows there is some $\presup{i}x \in \Gi^{b+ia}$
  which maps to $x \in \widetilde{G}^b$ for which $\presup{i}x$ is
  detected by $\presup{i}y \in \Ei^{s,b+ia}_r$ and $\presup{i+k+1}y$
  is nonzero for all $k\in\bbN$. Since $\presup{i}f$ is compatible
  with $\presup{i}g$, it follows that $\presup{i+k+1}y$ detects
  $\presup{i+k+1}x = \presup{i+k}g\circ\cdots \circ
  \presup{i}g(\presup{i}x)$
  for all $k\in \bbN$. Furthermore, $\presup{i+k}x\in\presup{i+k}G$ is
  nonzero for all $k\in \bbN$ and so $\presup{i+k}y$ survives to
  $\presup{i+k}E_{\infty}^{s,b+(i+k)a}$. Our assumption that the
  spectral sequences $\Ei$ converge strongly to $\Gi$ means that
  \begin{equation*}
    \presup{i+k}E_{\infty}^{s,b+(i+k)a} \cong \presup{i+k}\rmF_s / \presup{i+k}\rmF_{s+1}.
  \end{equation*}
  Hence every class $\presup{i+k}x$ is in filtration $s$ but not
  $s+1$, so that $x \in \widetilde{F}_s$ but
  $x \notin \widetilde{F}_{s+1}$. It now follows that the filtration
  $\widetilde{F}$ of $\widetilde{G}$ is Hausdorff. 
\end{proof}

\begin{lemma}
  \label{lem:existence}
  Under the conditions of proposition \ref{prop:ss_convergence},
  consider a nonzero element $x \in \widetilde{G}^b$. There exists
  some $\presup{i}x \in \Gi^{b+ia}$ that maps to
  $x \in \widetilde{G}^b$ for which $\presup{i}x$ is detected by
  $\presup{i}y \in \Ei^{s,b+ia}_r$ and
  \begin{equation*}
    \presup{i+k+1}y 
    = \presup{i+k}f\circ\cdots \circ \presup{i}f(\presup{i}y) \in \presup{i+k+1}E^{s, b+(i+k+1)a}_r
  \end{equation*}
  is nonzero for all $k\in\bbN$.
\end{lemma}

\begin{proof}
  There is some $\presup{j}x \in \presup{j}\rmG$ which maps to
  $x\in\widetilde{G}$. The classes
  $\presup{j+k}x=\presup{j+k-1}g\circ\cdots \circ
  \presup{j}g(\presup{j}x)$
  are non-zero for all $k\geq 1$ and are therefore detected by some
  class $\presup{j+k}y \in \presup{j+k}\rmE^{s_k,*}$. The vanishing
  line implies $s_k \leq N$ for $k$ sufficiently large and the
  compatibility of the maps $\presup{j}f$ with $\presup{j}g$ implies
  $s_k$ is a non-decreasing function of $k$. Hence $s_k$ is eventually
  constant, say for all $k+j \geq i$. Then $\presup{i}x$ has the
  desired property.
\end{proof}

These results can be applied to inverting multiplication by $h_1$ in
the motivic Adams spectral sequence at the prime $2$ after re-indexing
the filtration. For a field $F$, write $\Ei^f$ for
$\MASS(F)^{f+i}$. Here $f$ is the internal degree of the spectral
sequence (``f'' for ``filtration''). With this convention, the maps
$h_1 \cdot : \Ei \to \presup{i+1}\rmE$ form a directed system of
spectral sequences which is compatible with the maps
$\eta : \presup{i}\rmG \to \presup{i+1}\rmG$ where
$\presup{i}F_j = F_{j+i}(\pi_{**}(\unit\twocomp))$. The degree of
multiplication by $\eta$ is $(s,w)=(1,1)$ where $s$ is the stem and
$w$ the weight. A horizontal vanishing line of height $N$ in the
direction $(1,1)$ is equivalent to the following condition. For any
$(s,w)$ there exists $k$ so that for all $i >k$ and $f>N$ the group
$\MASS(F)^{f+i,s+i,w+i}$ vanishes. In the usual manner of drawing a
chart for $\MASS(F)$, such as those made by Isaksen \cite{Charts}, the
horizontal vanishing line for the system $\Ei^f$ is transformed into a
line of slope 1.

Such vanishing conditions occur over $\bbR$ in positive Milnor-Witt
stems as proved by Guillou and Isaksen \cite[Lemma 5.1]{GI-Real}. Over
$\bbR$ it suffices to take $N=1$, but one must take larger values for
other fields. For fields of cohomological dimension at most $2$ and
number fields, take $N=3$ for the positive Milnor-Witt stems.

\begin{corollary}
  \label{cor:convergence}
  The $h_1$-inverted motivic Adams spectral sequence over fields of
  cohomological dimension at most $2$ and the field of rational
  numbers $\bbQ$ converges strongly to
  $\pi_{s,w}(\unit\twocomp)[\eta^{-1}]$ when $s-w >1$.
\end{corollary}

\begin{proof}
  Consider the directed system $\Ei^f = \MASS(F)^{f+i}$ with maps
  $h_1 \cdot : \Ei \to \presup{i+1}\rmE$ as described above. With
  $N=3$, the vanishing conditions required for proposition
  \ref{prop:ss_convergence} are satisfied for fields $F$ of
  cohomological dimension at most 2. The $\rho$-Bockstein spectral
  sequence for such a field has $E_1$-page
  $H_{**}(F) \otimes_{\bbF_2[\tau]} \Ext(\bbC)$ and converges off to
  $\Ext(F)$. The $E_1$ page of the $\rho$-Bockstein spectral sequence
  has the claimed vanishing line in positive Milnor-Witt stem; hence
  $\Ext(F)$ does too. 

  An argument similar to the one given by Guillou and Isaksen in
  \cite[Lemma 5.1]{GI-Real} establishes a vanishing line over $\bbQ$
  in positive Milnor-Witt stems with $N=3$. Their choice of $A$ works
  just as well over $\bbQ$ ($A$ corresponds to $k$ when $s=0$ in the
  notation above) because the $\rho$-inverted Hopf algebroid
  $(H_{**}(\bbQ)[\rho^{-1}], \calA_{**}(\bbQ)[\rho^{-1}])$ is
  isomorphic to the $\rho$-inverted Hopf algebroid over $\bbR$. Their
  argument with the $\rho$-Bockstein spectral sequence must only be
  modified to account for $y$ being of the form $y = \alpha \tilde{y}$
  with $\tilde{y} \in \Ext(\bbC)$ and $\alpha\in H_{i,i}(\bbQ)$ with
  $i\leq 2$ and $\alpha$ is not divisible by $\rho$.

  The motivic Adams spectral sequence for fields $F$ with
  $\cd_2(F)\leq 2$ and $\bbQ$ converges conditionally to
  $\pi_{**}(\unit\twocomp)$ by \cite[Theorem 1]{HKO-Convergence} of
  Hu, Kriz, and Ormsby. The vanishing line described above ensures
  that it also converges strongly in Milnor-Witt stem at least 2, as
  in such degrees $d_r =0$ for $r$ sufficiently large. Hence we get
  the convergence result of the $h_1$-inverted Adams spectral
  sequences.
\end{proof}

\section{Fields of cohomological dimension at most 2}

Let $F$ be a field of 2-cohomological dimension at most 2. The mod 2
Milnor $K$-theory of such a field satisfies $k_n^M(F) = 0$ for
$n\geq 3$.  We first calculate $\Ext(F)[h_1^{-1}]$ using the
$\rho$-Bockstein spectral sequence and then observe that the structure
of $\MASS(F)_2 \cong \Ext(F)[h_1^{-1}]$ forces the $h_1$-inverted
motivic Adams spectral sequence to collapse at the $E_2$ page. See
figures \ref{fig:cd2BSS-E1} and \ref{fig:cd2BSS-Einfty} for a
depiction of the $\rho$-Bockstein spectral sequence $E_1$ and
$E_{\infty}$ page up to Milnor-Witt stem 24.

\begin{proposition}
  \label{prop:F-ext}
  For $F$ a field with $\cd_2(F) \leq 2$, the $E_2$ page of
  $\MASS(F)[h_1^{-1}]$ is
  \begin{equation*}
    \Ext(F)[h_1^{-1}] \cong k^M_*(F) 
    \otimes \Ext(\bbC)[h_1^{-1}].
  \end{equation*}
\end{proposition}

\begin{proof}
  If $-1$ is a square in $F$, it follows that
  $\Ext(F) \cong H_{**}(F)\otimes_{\bbF_2[\tau]} \Ext(\bbC)$ by an
  argument similar to \cite[Proposition 7.1]{WO}. The class $\tau$ is
  killed after inverting $h_1$, so the result follows in this case.

  If $-1$ is not a square in $F$, use the $\rho$-Bockstein spectral
  sequence. The $E_1$ page of the $h_1$-inverted $\rho$-Bockstein
  spectral sequence is 
  \begin{equation*}
    E_1^{\epsilon,*,*,*} \cong \rho^{\epsilon}k^M_*(F)/\rho^{\epsilon +1} k^M_*(F) \otimes \Ext(\bbC)[h_1^{-1}]
  \end{equation*}
  and the $d_r$ differential has degree $(r,1,-1,0)$ with respect to
  the grading $(\epsilon, f, t, c)$.

  The differentials $d_r$ with $r\geq 1$ vanish on the generators
  $P=v_1^{4}$ and $v_n$ for $n\geq 2$ of $\Ext(\bbC)[h_1^{-1}]$ for
  degree reasons. Any non-zero class $x \in \Ext(\bbC)[h_1^{-1}]$ has
  $t+c \equiv 0 \bmod 4$, but the degree of $d_r(x)$ satisfies
  $t+c \equiv 3 \bmod 4$. If $F$ has cohomological dimension at most
  2, then any non-zero class in the $h_1$-inverted $\rho$-Bockstein
  spectral sequence satisfies $t+c \not\equiv 3 \bmod 4$. Hence the
  $h_1$-inverted $\rho$-Bockstein spectral sequence collapses. There
  is no possibility for hidden extensions, so the proposition follows.
\end{proof}

\begin{proposition}
  \label{prop:alg_closed}
  If $\Fbar$ is an algebraically closed field of characteristic
  different from $2$, the $\eta$-inverted motivic homotopy groups of
  spheres over $\Fbar$ are given by
  \begin{equation*}
    \pi_{**}(\unit\twocomp)(\Fbar)[\eta^{-1}] \cong \pi_{**}(\unit\twocomp)(\bbC)[\eta^{-1}] \cong \bbF_2[\eta^{\pm 1}, \mu, \varepsilon]/(\varepsilon^2)
  \end{equation*}
  where $\mu \in \pi_{9,5}(\unit\twocomp)$ is the unique homotopy
  class detected by $Ph_1$ and
  $\varepsilon \in \pi_{8,5}(\unit\twocomp)$ is the unique homotopy
  class detected by $c_0$.
\end{proposition}

\begin{proof}
  If $\Fbar$ is characteristic zero, there is an isomorphism
  $\MASS(\Fbar) \cong \MASS(\bbC)$ by the proof of \cite[Lemma
  6.4]{WO}.  If $\Fbar$ has positive characteristic, the change of
  characteristic argument \cite[Corollary 6.1]{WO} comparing
  $\MASS(\Fbar)$ to $\MASS(\bbC)$ via the motivic Adams spectral
  sequence over the ring of Witt vectors of $\Fbar$ shows there is an
  isomorphism of spectral sequences $\MASS(\Fbar) \cong \MASS(\bbC)$.
  This isomorphism propagates to an isomorphism after inverting
  multiplication by $h_1$. The now resolved conjecture of Guillou and
  Isaksen in \cite[Conjecture 1.3]{GI} gives the explicit
  description.
\end{proof}

\begin{proposition}
  \label{prop:cd2-mass}
  The $d_2$ differentials for the $h_1$-inverted motivic Adams
  spectral sequence for a field $F$ with characteristic different from
  2 and $\cd_2(F)\leq 2$ follow from $d_2(v_n) = h_1^2v_{n-1}^2$ for
  $n\geq 3$ and $d_2(x) = 0 $ for $x \in k^M_*(F)$ by using the
  Leibniz rule. Furthermore, $\MASS(F)[h_1^{-1}]$ collapses at the
  $E_3$ page.
\end{proposition}

\begin{proof}
  The inclusion of $F$ into its algebraic closure $\Fbar$ induces a
  map of spectral sequences
  \begin{equation*}
    \Phi : \MASS(F)[h_1^{-1}] \to \MASS(\Fbar)[h_1^{-1}] \cong \MASS(\bbC)[h_1^{-1}].
  \end{equation*}
  Andrews and Miller \cite[Theorem 9.15]{AM} have proved that in
  $\MASS(\bbC)[h_1^{-1}]$ there are differentials
  $d_2(v_n)=h_1^2v_{n-1}^2$ for all $n\geq 3$. It follows that in
  $\MASS(F)[h_1^{-1}]$ we must have $d_2(v_n)= h_1^2v_{n-1}^2$ up to
  some element in the kernel of the comparison map $\Phi$. A class
  $x \in \ker(\Phi)$ satisfies $t+c \equiv 1 \bmod 4$ or
  $t+c \equiv 2 \bmod 4$, whereas $d_2(v_n)$ satisfies
  $t+c \equiv 0 \bmod 4$. Hence $d_2(v_n)= h_1^2v_{n-1}^2$ is true on
  the nose.  That the spectral sequence collapses at the $E_3$ page
  follows by degree reasons.
\end{proof}

\begin{theorem}
  \label{thm:cd2}
  For a field $F$ with $\cd_2(F)\leq 2$ and characteristic different
  from 2 the two-complete $\eta$-inverted Milnor-Witt stems of $F$ are
  \begin{equation*}
    \Pihat_t(F)[\eta^{-1}] \cong
    \begin{cases}
      W(F)\twocomp[\eta^{\pm 1}] & t\geq 0 \text{ and } t\equiv 3 \bmod 4 \text{ or } t \equiv 0 \bmod 4 \\
      0 & \text{otherwise.}
    \end{cases}
  \end{equation*}
  $\Pihat_*(F)[\eta^{-1}]$ is the polynomial ring over
  $W(F)\twocomp[\eta^{\pm 1}]$ on two classes $\{v_2\}$ and $\{P\}$
  in Milnor-Witt stem $3$ and $4$ respectively, subject to the
  relation $\{v_2\}^2 = 0$.
\end{theorem}

\begin{proof}
  $\Pihat_0(F)[\eta^{-1}]$ is shown to be $W(F)\twocomp[\eta^{\pm 1}]$
  in proposition \ref{prop:0-stem}. The remaining stems and ring
  structure follows from the calculation of the $h_1$-inverted motivic
  Adams spectral sequence over $F$ whose differentials are determined
  in proposition \ref{prop:cd2-mass}.
\end{proof}

We now identify some classes in $\pi_{**}(\unit\twocomp)(\bbF_q)$ for
finite fields $\bbF_q$ using the analysis of the motivic Adams
spectral sequence by Wilson and {\O}stv{\ae}r in \cite{WO}. Over a
finite field $\bbF_q$ with $q\equiv 1 \bmod 4$, define
$\varepsilon \in \pi_{8,5}(\unit\twocomp) \cong (\bbZ/2)^4$ to be a
class detected by $c_0$. The class $\varepsilon$ is uniquely
determined modulo $u\eta \varepsilon$. Write $\mu$ for a class in
$\pi_{9,5}(\unit\twocomp)\cong (\bbZ/2)^4$ detected by $Ph_1$. The
class $\mu$ is uniquely determined modulo $u\eta \mu$.

Over a finite field $\bbF_q$ with $q\equiv 3 \bmod 4$, there is an
isomorphism $\pi_{8,5}(\unit\twocomp)\cong \bbZ/4\oplus \bbZ/4$.
Recall there is a Hopf map $\sigma \in \pi_{7,4}(\unit\twocomp)$
defined by Dugger and Isaksen in \cite{DI-Hopf}. The class
$\eta \sigma$ generates an order four cyclic subgroup of
$\pi_{8,5}(\unit\twocomp)$; define
$\varepsilon\in \pi_{8,5}(\unit\twocomp)$ by the property that
$\varepsilon$ generates $\pi_{8,5}(\unit\twocomp)/(\eta\sigma)$. The
class $\varepsilon$ is detected by $c_0$ and well defined up to an odd
multiple. Further, there is an isomorphism
$\pi_{9,5}(\unit\twocomp) \cong \bbZ/4\oplus(\bbZ/2)^2$. Define $\mu$
to be a class of order four that is detected by $Ph_1$; the class
$\mu$ is uniquely defined up to an odd multiple.

\begin{corollary}
  For a finite field $\bbF_q$ with $q$ odd, the $\eta$-inverted
  Milnor-Witt stems are as follows.
  \begin{equation*}
    \Pihat_n(\bbF_q)[\eta^{-1}] \cong 
    \begin{cases}
      W(\bbF_q)\twocomp[\eta^{\pm 1}] & n\geq 0 \text{ and } n\equiv 3 \bmod 4 \text{ or } n \equiv 0 \bmod 4 \\
      0 & \text{otherwise}
      \end{cases}
  \end{equation*}
  The classes $\mu$ and $\varepsilon$ generate
  $\Pihat_*(\bbF_q)[\eta^{-1}]$ as an algebra over
  $\Pihat_0(\bbF_q)[\eta^{-1}]$, subject to the relation
  $\varepsilon^2 = 0$. When $q\equiv 3\bmod 4$ this shows
  $\Pihat_*(\bbF_q) \cong \bbZ/4[\eta^{\pm 1}, \mu,\epsilon]/\epsilon^2$
  and for $q \equiv 1 \bmod 4$ we have
  $\Pihat_*(\bbF_q) \cong \bbZ/2[\eta^{\pm 1}, u,
  \mu,\epsilon]/(u^2, \epsilon^2)$.
\end{corollary}

\begin{proof}
  The mod 2 Milnor $K$-theory of a finite field with odd
  characteristic is given by $k^M_*(\bbF_q) = \bbF_2[u]/u^2$ where $u$
  is the class of a non-square element of $\bbF_q^{\times}$. If
  $q \equiv 3 \bmod 4$ then $u = \rho = [-1]$. As $h_1\rho$ in
  $\MASS(\bbF_q)[h_1^{-1}]$ detects multiplication by 2 in
  $\pi_{**}(\unit\twocomp)$, we arrive at the claimed group
  structure. The product structure is clear given the products in the
  $h_1$-inverted motivic Adams spectral sequence.
\end{proof}

\begin{corollary}
  The $\eta$-inverted Milnor-Witt stems for a $p$-adic field $\bbQ_p$
  are as follows. 
  \begin{equation*}
    \Pihat_0(\bbQ_p)[\eta^{-1}] \cong W(\bbQ_p)\twocomp[\eta^{\pm 1}] \cong
    \begin{cases}
      \bbZ/2[\eta^{\pm 1}, u, \pi ]/(u^2, \pi^2) &  p \equiv 1 \bmod 4 \\
      (\bbZ/4\oplus \bbZ/4)[\eta^{\pm 1}] & p \equiv 3 \bmod 4 \\
      (\bbZ/8 \oplus \bbZ/2 \oplus \bbZ/2)[\eta^{\pm 1}] & p = 2 
    \end{cases}
  \end{equation*}

  \begin{equation*}
    \Pihat_n(\bbQ_p)[\eta^{-1}] \cong 
    \begin{cases}
      W(\bbQ_p)\twocomp[\eta^{\pm 1}] & n\geq 0 \text{ and } n\equiv 3 \bmod 4 \text{ or } n \equiv 0 \bmod 4 \\
      0 & \text{otherwise}
      \end{cases}
  \end{equation*}
\end{corollary}

\begin{proof}
  The mod 2 Milnor $K$-theory of the $p$-adic fields can be calculated
  from the result of Milnor \cite[Lemma 4.6]{Milnor-KTheory} in
  addition to the description of the Witt ring for $p$-adic fields
  which is discussed by Serre in \cite{Serre-Arithmetic}. Explicitly,
  the mod 2 Milnor $K$-theory of a $p$-adic field is
  \begin{equation*}
    k^M_*(\bbQ_p) =  
    \begin{cases}
      \bbZ/2[\pi, u]/(\pi^2, u^2) & p \equiv 1 \bmod 4 \\
      \bbZ/2[\pi,\rho]/(\rho^2, \rho\pi+\pi^2) & p \equiv 3\bmod 4\\
      \bbZ/2[\pi, \rho, u]/(\rho^3, u^2, \pi^2, \rho u, \rho \pi, \rho^2 + u \pi) & p = 2
    \end{cases}
  \end{equation*}
  where $\pi=[p]$, $\rho = [-1]$, $u$ is the class of a lift of a
  non-square in $\bbF_p^{\times}$ when $p \equiv 1 \bmod 4$, and
  $u=[5]$ when $p = 2$.
\end{proof}

\section{The field of rational numbers}
\label{sec:Q} 

We approach the calculation of $\Pihat_*(\bbQ)[\eta^{-1}]$ with the
strategy suggested by the motivic Hasse principle, following the
method of Ormsby and {\O}stv{\ae}r in \cite{MotBPInvQ}. That is, we
analyze the $h_1$-inverted motivic Adams spectral sequence for $\bbQ$
using our knowledge of the $h_1$-inverted motivic Adams spectral
sequence over the completions of $\bbQ$.

We fix our notation for $k^M_*(\bbQ)$. The mod 2 Milnor $K$-theory of
$\bbQ$ is generated by the classes $[-1]$ and $[p]$ for $p$ a
prime. Milnor shows in \cite[Lemma A.1]{Milnor-KTheory} that there is
a short exact sequence
\begin{equation*}
  0 \to k_2^M(\bbQ) \to \bigoplus k_2^M(\bbQ_{\nu}) \to \bbZ/2 \to 0
\end{equation*}
where the summation is over all completions $\bbQ_{\nu}$ of
$\bbQ$. For every completion $\bbQ_{\nu}$ of $\bbQ$ there is an
isomorphism $k_2^M(\bbQ_{\nu})\cong \bbZ/2$; write $e_{\nu}$ for the
image of $1$ under the canonical map
$k_2^M(\bbQ_{\nu}) \to \bigoplus k_2^M(\bbQ_{\nu})$. For $\ell$ an odd
prime or $-1$, write $a_{\ell}$ for the class in $k_2^M(\bbQ)$ that
maps to $e_{\ell} + e_2$ in $\bigoplus k_2^M(\bbQ_{\nu})$. For
$n\geq 3$ the class $\rho^n$ generates $k_n^M(\bbQ)$. The product
structure in $k_*^M(\bbQ)$ can be deduced from the products given in
table \ref{table:Q-products}; we write $\left(\dfrac{q}{\ell}\right)$
for the Legendre symbol that takes values in the additive group
$\bbZ/2$.

\begin{table}[h]
  \centering
  \begin{tabular}{l|l}
Product & Conditions \\ \hline
    $[\ell][2] = 0 $ & $\ell = 2$ or $\ell = -1$ \\ 
    $[-1][\ell] =a_{\ell}$ & $\ell = -1$ or $\ell$ prime and  $\ell \equiv 3\bmod 4$ \\
    $[-1][\ell] = 0$ & $\ell$ prime and $\ell \equiv 1 \bmod 4$ \\ 
    $[\ell][q] = \left(\dfrac{q}{\ell}\right)a_{\ell} + \left(\dfrac{\raisebox{-3pt}{$\ell$}}{\raisebox{3pt}{$q$}}\right)a_{q}$ & $\ell$ and $q$ odd primes \\ 
    $[2][q] = (\frac{q^2-1}{8} \bmod 2)a_q$ & $q$ odd prime \\
  \end{tabular}
  \caption{Products in $k_*^M(\bbQ)$}
  \label{table:Q-products}
\end{table}

\begin{lemma}
  \label{lem:Q-BSS}
  The $E_1$ page of the $h_1$-inverted $\rho$-Bockstein spectral
  sequence over $\bbQ$ is the $\bbZ/2$-algebra
  \begin{equation*}
    \BSS(\bbQ)[h_1^{-1}] \cong \bigoplus_{n\in\bbN} \rho^n k^M_*(\bbQ) 
    / \rho^{n+1} \otimes_{\bbF_2} \Ext(\bbC)[h_1^{-1}].
  \end{equation*}
  The class $\rho^n$ is in filtration $\epsilon = n$ for all
  $n\in\bbN$, for $\ell \equiv 3 \bmod 4$ a prime $a_{\ell}$ is in
  filtration 1, for $\ell \equiv 1 \bmod 4$ a prime $a_p$ is in
  filtration 0, and $[p]$ for $p$ a prime is in filtration 0. The
  $r$th differential $d_r$ for the $\rho$-Bockstein spectral sequence
  has degree $(\epsilon, f, t, c) = (r, 1, -1, 0)$. See figure
  \ref{fig:Q-RhoBSS} for a chart of the $E_1$-page up to Milnor-Witt
  stem 15 .
\end{lemma}

\begin{proof}
  The $\rho$-Bockstein spectral sequence arises from filtering the
  cobar complex $\calC^*(\bbQ)$ by powers of $\rho$. The $s$th term of
  the cobar complex is
  $\calC^s(\bbQ) = H_{**}(\bbQ)\otimes \calA_{**}(\bbQ)^{\otimes s}$,
  where the tensor products are taken over $H_{**}(\bbQ)$, taking care
  to use the left and right actions of $H_{**}(\bbQ)$ on
  $\calA_{**}(\bbQ)$ arising from the left and right units $\eta_L$
  and $\eta_R$. Any class $a[x_1\vert \cdots \vert x_s]$ can be
  reduced to a sum of monomials $b[y_1 \vert \cdots \vert y_s]$ where
  each $y_i$ is a monomial in
  $\bbZ/2[\tau_0,\tau_1,\ldots, \xi_1, \xi_2,\ldots]$.  The class
  $\tau$ is killed after inverting $h_1$, hence every element of
  $\calC^s(\bbQ)$ is a sum of monomials
  $b[y_1 \vert \cdots \vert y_s]$ where each $y_i$ is a monomial in
  $\bbZ/2[\tau_0,\tau_1,\ldots, \xi_1, \xi_2,\ldots]$ and
  $b \in k^M_*(\bbQ)$. The filtration of the cobar complex now is
  determined by the filtration of $k^M_*(\bbQ)$ by powers of $\rho$.
\end{proof}

\begin{proposition}
  The differentials for the $h_1$-inverted $\rho$-Bockstein spectral
  sequence over $\bbQ$ are determined by
  $d_{2^n-1}(v_1^{2^n}) = h_1^{2^n}\rho^{2^n-1}v_n$ for $n\geq 2$ and
  $d_r(v_n) = 0$ for $r\geq 1$ and $n\geq 2$. See figure
  \ref{fig:Q-RhoBSSInfinity} for a chart of the $E_{\infty}$ page up
  to Milnor-Witt stem 15.
\end{proposition}

\begin{proof}
  The injection $k_*^M(\bbQ) \to \prod_{\nu} k_*^M(\bbQ_{\nu})$
  extends to an injection of $h_1$-inverted $\rho$-Bockstein spectral
  sequences at the $E_1$ page.
  \begin{equation*}
    \BSS(\bbQ)[h_1^{-1}] \to \prod_{\nu} \BSS(\bbQ_{\nu})[h_1^{-1}]
  \end{equation*}
  The differentials in the $h_1$-inverted $\rho$-Bockstein spectral
  sequence over $\bbQ_p$ vanish for all primes $p$. Only the
  differentials in $\BSS(\bbR)[h_1^{-1}]$ contribute to the
  differentials over $\bbQ$, and these were identified by Guillou and
  Isaksen in \cite[Lemma 3.1]{GI-Real}.
\end{proof}

\begin{proposition}
  The $h_1$-inverted $\rho$-Bockstein spectral sequence for $\bbQ$
  converges strongly to $\Ext(\bbQ)[h_1^{-1}]$ and there are no hidden
  extensions. 
\end{proposition}

\begin{proof}
  The $h_1$-inverted $\rho$-Bockstein spectral sequence is isomorphic
  to the $\rho$-Bockstein spectral sequence obtained by filtering the
  $[\xi_1]$-inverted cobar complex $\calC^*(\bbQ)$, hence it converges
  strongly to $\Ext(\bbQ)[h_1^{-1}]$. Guillou and Isaksen have shown
  that there are no hidden extensions in the $h_1$-inverted
  $\rho$-Bockstein spectral sequence over $\bbR$ \cite[Proposition
  4.9]{GI-Real} and there are no hidden extensions in the
  $h_1$-inverted $\rho$-Bockstein spectral sequence over the other
  completions of $\bbQ$ by proposition \ref{prop:F-ext}. We therefore
  conclude there are no hidden extensions since the Hasse map embeds
  $\BSS(\bbQ)[h_1^{-1}]_{\infty} \Rightarrow \Ext(\bbQ)[h_1^{-1}]$
  into
  $\prod_{\nu} \BSS(\bbQ_{\nu})[h_1^{-1}]_{\infty} \Rightarrow
  \prod_{\nu} \Ext(\bbQ_{\nu})[h_1^{-1}]$.
\end{proof}

\begin{corollary}               
  \label{cor:Q-ext}
  $\Ext(\bbQ)[h_1^{-1}]$ is generated by the classes in table
  \ref{table:Q-ext-gens}. The relations among these generators over
  $k^M_*(\bbQ)$ include: $[\ell]P^k \cdot [q]P^j =
  [\ell]\cdot[q]P^{k+j}$ for $\ell$, $q$ primes and $k,j \geq 0$;
  $[\ell]P^{2^{n-1}k}\cdot v_n = [\ell]\cdot P^{2^{n-1}k}v_n$ for
  $\ell$ a prime, $n\geq 2$, and $k\geq 0$; the vanishing of the
  product of three or more generators of the form $[\ell]P^k$; and the
  relations which set the $\rho$-torsion of the generators.
\begin{table}[h]
  \centering
  \begin{tabular}{c|c|c|l}
    Class & $(f, t, c)$ & $\rho$-torsion & Conditions \\\hline
    $h_1^{\pm 1}$ & $(\pm 1,0,0)$ & $\infty$ \\
    $\rho$ & $(0, 0, 1)$ & $\infty$ &\\
    $[2]+\rho$ & $(0,0,1)$ & 1  & \\
    $[\ell]P^k$ & $(0,0,1)+k(4,4,4)$ & 1 & $\ell$ prime, $\ell \equiv 1\,(4)$, $k\geq 0$\\
    $[\ell]P^k$ & $(0,0,1)+k(4,4,4)$ & 2  &$\ell$ prime, $p\equiv 3\,(4)$, $k\geq 0$ \\     $[2]P^k$ & $(0,0,1)+k(4,4,4)$ & 1 & $k\geq 1$ \\
    $P^{2k} v_2$ & $(1,3,1) + k(8,8,8)$ & 3 & $k\geq 0$ \\
    $P^{4k}v_3$ & $(1,7,1) + k(16,16,16)$  & 7 & $k\geq 0$ \\
    $P^{8k}v_4$ & $(1,15,1) + k(32,32,32)$ & 15 & $k\geq 0$ \\
    $\cdots$ & $\cdots$ & $\cdots$ & $\cdots$
  \end{tabular} 
  \caption{Generators of $\Ext(\bbQ)[h_1^{-1}]$.}
  \label{table:Q-ext-gens}
\end{table}
\end{corollary}

\begin{proof}
  The generators can be determined by comparing $\Ext(\bbQ)[h_1^{-1}]$
  to $\Ext(\bbR)[h_1^{-1}]$, and the latter was determined by Guillou
  and Isaksen in \cite[Theorem 4.10]{GI-Real}. The relations stated
  are present in the $\rho$-Bockstein spectral sequence and persist to
  $\Ext(\bbQ)$.
\end{proof}

The differentials in $\MASS(\bbQ)[h_1^{-1}]$, the $h_1$-inverted Adams
spectral sequence over $\bbQ$, are determined by the differentials
obtained from the comparison to $\bbQ_p$ and $\bbR$.  See figures
\ref{fig:Q-RhoBSSInfinity} and \ref{fig:Q-E3} for a depiction of the
$E_2$ and $E_3$ pages up to Milnor-Witt stem 15.

\begin{proposition}
  \label{prop:Q-Einfty}
  The $d_2$ differential in $\MASS(\bbQ)[h_1^{-1}]$ is determined by
  the Leibniz rule from the equations $d_2(P^{2^{n-1}k}v_n) =
  P^{2^{n-1}k}v_{n-1}^2$ for $k\geq 0$ and $n\geq 3$ and the vanishing
  of $d_2$ on the remaining generators. For $r\geq 3$, the
  differential $d_r$ in $\MASS(\bbQ)[h_1^{-1}]$ is determined by the
  Leibniz rule from the equations
  \begin{equation*}
    d_r(\rho^{2^n-2^{n-r+2}-r+2}P^{2^{n-1}k}v_n) = P^{2^{n-1}k+2^{n-2}-2^{n-r}} v_{n-r+1}^2
  \end{equation*}
  for $n \geq r+1$ and the vanishing of $d_r$ on the remaining
  generators.
\end{proposition}

\begin{proof}
  $\Ext(\bbQ)[h_1^{-1}]$ injects into the product
  $\prod_{\nu}\Ext(\bbQ_{\nu})[h_1^{-1}]$ under the base change maps
  obtained from $\bbQ \to \bbQ_{\nu}$. The map is seen to be injective
  by the explicit calculation of $\Ext(\bbQ)[h_1^{-1}]$ given in
  corollary \ref{cor:Q-ext}, $\Ext(\bbQ_p)[h_1^{-1}]$ in proposition
  \ref{prop:F-ext}, and $\Ext(\bbR)[h_1^{-1}]$ in \cite[Theorem
  4.10]{GI-Real}. The differentials $d_2(v_n) = v_{n-1}^2$ for $n\geq
  3$ over $\bbQ_p$ imply that the class $d_2(P^{2^{n-1}k}v_n)$ must
  map to $ d_2(P^{2^{n-1}k}v_n) = P^{2^{n-1}k}v_{n-1}^2$ in
  $\Ext(\bbQ_p)[h_1^{-1}]$. Comparison to $\bbR$ also shows that the
  differential $d_2(P^{2^{n-1}k}v_n)$ maps to $P^{2^{n-1}k}v_{n-1}^2$
  in $\Ext(\bbR)[h_1^{-1}]$ for $n\geq 3$, as determined by Guillou
  and Isaksen \cite[Lemma 5.2]{GI-Real}. The differential $d_2$ over
  $\bbQ$ vanishes on the classes $[\ell]P^k$ by comparison to $\bbQ_p$
  for all $p$. Finally, $d_2$ vanishes on all elements of
  $k_*^M(\bbQ)$ for degree reasons.  This accounts for all of
  irreducible classes of $\Ext(\bbQ)[h_1^{-1}]$; the generators for
  $\MASS(\bbQ)_3[h_1^{-1}]$ are given in table \ref{table:Q-e3}.  Note
  that the classes $P^{2(2j+1)}v_2^2$ also survive but decompose as
  the product $P^{2(2j+1)}v_2 \cdot v_2$.
\begin{table}[h]
  \centering
  \begin{tabular}{c|c|c|l}
    Class & $(f, t, c)$ & $\rho$-torsion & Conditions \\\hline
    $h_1^{\pm 1}$ & $(\pm 1,0,0)$ & $\infty$ \\
    $\rho$ & $(0, 0, 1)$ & $\infty$ &\\
    $[2]+\rho$ & $(0, 0, 1)$ & 1 &\\
    $[\ell]P^k$ & $(0,0,1) +k(4,4,4)$ & 1 & $\ell$ prime,  $\ell\equiv 1\,(4)$, $k\geq 0$ \\
    $[\ell]P^k$ & $(0,0,1)+k(4,4,4)$ & 2  &$\ell$ prime, $\ell\equiv 3\,(4)$, $k\geq 0$\\ 
    $[2]P^k$ & $(0,0,1)+k(4,4,4)$ & 1  & $k\geq 1$\\
    $P^{2k} v_2$ & $(1,3,1) +k(8,8,8)$ & 3 & $k\geq 0$ \\
    $\rho^3 P^{4k}v_3$ & $(1,7,4)+k(16,16,16)$  & 4 & $k\geq 0$ \\
    $\rho^7 P^{8k}v_4$ & $(1, 15, 8)+k(32,32,32)$ & 8 & $k\geq 0$ \\
    $\cdots$ & $\cdots$&$\cdots$ & $\cdots$\\ 
    $P^{4(2j+1)}v_3^2$ & $(2, 14, 2) + (2j+1)(4,4,4)$ & $7$ & $j\geq 0$ \\ 
    $P^{8(2j+1)}v_4^2$ & $(2,30,2) + (2j+1)(4,4,4) $ & $15$ & $j\geq 0$ \\
    $\cdots$ & $\cdots$ & $\cdots$ & $\cdots$ \\
  \end{tabular} 
  \caption{Generators of $\MASS(\bbQ)[h_1^{-1}]_3$.}
  \label{table:Q-e3}
\end{table}

  The Hasse map
  \begin{equation*}
    \MASS(\bbQ)_3[h_1^{-1}] \to \prod_{\nu} \MASS(\bbQ_{\nu})_3[h_1^{-1}]
  \end{equation*}
  is still injective. Over the $p$-adic fields, all further
  differentials vanish, and over $\bbR$ the differentials are
  determined by Guillou and Isaksen \cite[Lemma 5.8]{GI-Real}; these
  comparisons determine the remaining differentials.
\end{proof}

\begin{proposition}
  \label{prop:generators}
  $\MASS(\bbQ)[h_1^{-1}]_{\infty}$ is generated over $k_*^M(\bbQ)$ by
  the classes $\rho^{2^n - n - 2} P^{2^{n-1}k}v_n$ for $k\geq 0$ and
  $n\geq 2$, which has degree
  $(2^{n+1}k+1, 2^{n+1}k+2^n-1, 2^{n+1}k+2^n-n-1)$ and $\rho$-torsion
  $n+1$. See table \ref{table:Q-generators} for some low degree
  generators and figure \ref{fig:Q-EInfinity} for a chart of the
  $E_{\infty}$ page up to Milnor-Witt stem 15.
\end{proposition}

\begin{proof}
  This is a consequence of the differential analysis of proposition
  \ref{prop:Q-Einfty} and the result of Guillou and Isaksen
  \cite[Proposition 5.9]{GI-Real}.
\end{proof}

\begin{table}[h]
  \centering
  \begin{tabular}{c|c|c|c}
    Class & $(f, t, c)$ & $\rho$-torsion & Conditions \\\hline
    $h_1^{\pm 1}$ & $(\pm 1,0,0)$ & $\infty$ \\
    $\rho$ & $(0, 0, 1)$ & $\infty$ &\\
    $[2] + \rho$ & $(0, 0, 1)$ & $1$ &\\
    $[\ell]P^k$ & $(0,0,1)+k(4,4,4)$ & 1 &$\ell$ prime,  $\ell\equiv 1\,(4)$, $k\geq 0$ \\
    $[\ell]P^k$ & $(0,0,1)+k(4,4,4)$ & 2  &$\ell$ prime,$\ell\equiv 3\,(4)$, $k\geq 0$ \\
    $[2]P^k$ & $(0,0,1)+k(4,4,4)$ & 1  &$k\geq 1$\\
    $P^{2k} v_2$ & $(1, 3, 1) + k(8,8,8)$ & 3 & $k\geq 0$ \\
    $\rho^3 P^{4k}v_3$ & $(1,7,4)+k(16,16,16)$  & 4 & $k\geq 0$ \\
    $\rho^{10} P^{8k} v_4$ & $(1, 15, 11) + k(32,32,32)$ & 5 & $k\geq 0$ \\
    $\cdots$ & $\cdots $ & $\cdots$ & $\cdots$ \\
  \end{tabular} 
  \caption{Generators of $\mathfrak{M}(\bbQ)[h_1^{-1}]_{\infty}$.}
  \label{table:Q-generators}
\end{table}

\begin{definition}
  \label{def:M}
  Let $M$ be the submodule of $W(\bbQ)\twocomp[\eta^{\pm 1}]$
  generated by the rank one forms $\ell\cdot X^2$ for $\ell$ a prime.
  As an abelian group, $M$ is isomorphic to
  $\bbZ/2 \oplus \bigoplus_{p \equiv 3\,(4)} \bbZ/4 \oplus
  \bigoplus_{p\equiv 1\,(4)} (\bbZ/2)^2[\eta^{\pm 1}]$.
\end{definition}

Following the notational convention of Isaksen and Guillou
\cite[\S7]{GI-Real}, write $P^{2^{n-1}k}\lambda_n$ for a class in
$\Pihat_{t}(\bbQ)[\eta^{-1}]$ detected by
$\rho^{2^n-n-2}P^{2^{n-1}k}v_n$ in $\MASS(\bbQ)[h_1^{-1}]$ where $n
\geq 2$, $k\geq 0$ and $t = 2^{n+1}k+2^n-1$. Also, we abuse notation
and write $[\ell]P^k$ for a class in $\Pihat_{4k}(\bbQ)[\eta^{-1}]$
detected by the class of the same name in $\MASS(\bbQ)[h_1^{-1}]$.
  
\begin{theorem}
  \label{thm:Q}
  The $\eta$-inverted Milnor-Witt 0 stem of $\unit\twocomp$ over
  $\bbQ$ is
  \begin{equation*}
    \Pihat_0(\bbQ)[\eta^{-1}] \cong W(\bbQ)\twocomp[\eta^{\pm 1}].
  \end{equation*}
  The $t$th $\eta$-inverted Milnor-Witt stem of $\unit\twocomp$ over
  $\bbQ$ are as follows.
  \begin{equation*}
        \Pihat_t(\bbQ)[\eta^{-1}] \cong 
        \begin{cases}
          \Pihat_0(\bbQ)[\eta^{-1}]/2^{n+1} 
          & t \geq 0, t \equiv 3 \bmod 4, n = \nu_2(t+1) \\
          M & t \equiv 0 \bmod 4, t \geq 4 \\
          0 & \text{otherwise}
        \end{cases}
  \end{equation*}
  Here $\nu_2(x)$ is the $2$-adic valuation of an integer $x$ and $M$
  is the $\Pihat_0(\bbQ)[\eta^{-1}]$-module of definition
  \ref{def:M}.

  The remaining product structure of $\Pihat_*(\bbQ)[\eta^{-1}]$ is
  determined by the following relations: the product of any two
  generators with Milnor-Witt stem congruent to $3 \bmod 4$ is zero;
  $[q]P^j \cdot [\ell] P^k = [q] \cdot [\ell]P^{j+k}$ for all primes
  $\ell$ and $q$ and $k,j \geq 0$; $[q]\cdot[\ell]P^k = 0 $ if $q$ is
  a prime or $-1$ and $[q][\ell]=0$ in $k^M_*(\bbQ)$.
\end{theorem}

\begin{proof}
  The zero stem was calculated in proposition \ref{prop:0-stem} and
  \cite[Proof of Theorem 1.5]{Vanishing} shows the one stem vanishes.
  Proposition \ref{prop:generators} identifies the structure of the
  $E_{\infty}$ page of the $h_1$-inverted motivic Adams spectral
  sequence over $\bbQ$ and corollary \ref{cor:convergence} shows that
  $\MASS(\bbQ)[h_1^{-1}]$ strongly converges to
  $\Pihat_*(\bbQ)[\eta^{-1}]$ in Milnor-Witt stem at least two.  The
  $2$-extensions are resolved because $\rho h_1$ detects
  multiplication by $2$, from which the additive structure of the
  $\eta$-inverted stems follows.

  The product structure in the $E_{\infty}$ page of the $h_1$-inverted
  motivic Adams spectral sequence determines the
  $\Pihat_0(\bbQ)[\eta^{-1}]$-module structure of the stems
  $\Pihat_t(\bbQ)[\eta^{-1}]$ for $t \not\equiv 3 \bmod 4$ and $t \equiv
  3 \bmod 8$. It only remains to identify the hidden product of
  $[\ell]$ for $\ell$ a prime with a class $P^{2^{n-1}k}\lambda_n$ of
  $\Pihat_t(\bbQ)[\eta^{-1}]$ for $u \geq 3$, $k\geq 0$; note that $n =
  \nu_2(t+1)$. Lemma \ref{lem:massey_products} shows that the products
  $[\ell] \cdot P^{2^{n-1}k}\lambda_n$ and $a_{\ell}\cdot
  P^{2^{n-1}k}\lambda_n$ are always non-zero; hence the canonical map
  $\Pihat_0(\bbQ)[\eta^{-1}]/2^{n+1} \to \Pihat_n(\bbQ)[\eta^{-1}]$ is an
  isomorphism.

  The product of any two generators with Milnor-Witt stem congruent to
  $3 \bmod 4$ is zero for degree reasons. The remaining products are
  detected in the motivic Adams spectral sequence.
\end{proof}

\begin{lemma}
  \label{lem:massey_products}
  For $n \geq 3$ and $k\geq 0$, the products $[\ell] \cdot
  P^{2^{n-1}k}\lambda_n$ with $\ell$ a prime and $a_{\ell}\cdot
  P^{2^{n-1}k}\lambda_n$ with $\ell$ an odd prime in
  $\Pihat_*(\bbQ)[\eta^{-1}]$ are non-zero.
\end{lemma}

\begin{proof}
  For $m\geq 0$ and $\ell$ a prime, the Massey product $\langle \rho
  P^{2m}v_2, \rho^2, [\ell]\rangle$ in $\Ext(\bbQ)[h_1^{-1}]$ contains
  $[\ell]P^{2m+1}$ by lemma \ref{lem:may_convergence} and this has no
  indeterminacy. The hypotheses of Moss's convergence theorem
  \cite[Theorem 3.1.1]{StableStems} hold in this
  case;\footnote{Observe that for $r\geq 2$, $c>t$ and $t\equiv 3
    \bmod 4$ that the groups $E_2^{t,c}$ are trivial in the
    $h_1$-inverted motivic Adams spectral sequence over $\bbQ$.} hence
  $[\ell]P^{2m+1}$ detects a class of $\langle 2 P^{2m} \lambda_2,
  2^2, [\ell]\rangle$. The indeterminacy of this Toda bracket is
  $[\ell]\Pihat_{8m+4}$, which is in higher filtration than
  $[\ell]P^{2m+1}$. We conclude $\langle 2 P^{2m} \lambda_2, 2^2,
  [\ell]\rangle$ does not contain zero.

  The Massey product $\langle v_2, \rho P^{2m}v_2, \rho^2 \rangle$ can
  be shown to contain $\rho^{2^n-n-2}P^{2^{n-1}k}v_n$ using the Adams
  differential
  \begin{equation*}
    d_r(\rho^{2^n - n - 4}P^{2^{n-1}k}v_n) = \rho P^{2m}v_2^2
  \end{equation*}
  where $n = \nu_2(m+1) + 3$ and $r = n-1$. This Massey product has
  trivial indeterminacy. Moss's convergence theorem shows that
  $\rho^{2^n-n-2}P^{2^{n-1}k}v_n$ detects a class of $\langle
  \lambda_2, 2 P^{2m}\lambda_2, 2^2 \rangle$;\footnote{The condition
    to check for $\rho P^{2m}v_2^2$ in Moss's theorem is that
    $d_{r'}:E_{r'}^{8m+7, c'} \to E_{r'}^{8m+6, c''}$ must be zero
    when $c' \leq 8m + 1 - \nu_2(m+1)$, $c'' \geq 8m+4$ and $r' = c''
    - c'$. Nonzero differentials are only possible in these
    Milnor-Witt stems on classes $\rho^?P^?v_n$; it follows that
    $\nu_2(m+1)=n-3$, so $r' \geq n$. But by proposition
    \ref{prop:Q-Einfty}, nonzero differentials on such classes occur
    only when $n \geq r' + 1$. The condition to check for the
    element $\rho^3P^{2m}v_2$ is that $d_{r'}:E_{r'}^{8m+4, c'} \to
    E_{r'}^{8m+3, c''}$ must be zero when $c' \leq 8m+2-\nu_2(m+1)$,
    $c'' \geq 8m + 5$ and $r' = c'' - c'$. The classes in Milnor-Witt
    stem $4 \bmod 8$ are generated by the classes of the form $\rho$,
    $h_1^{\pm 1}$ and $[\ell]P^k$. The Adams differentials vanish on
    these classes, so the hypotheses of Moss's theorem are true.}
  hence $P^{2^{n-1}k}\lambda_n$ is in the Toda bracket $\langle
  \lambda_2, 2 P^{2m}\lambda_2, 2^2 \rangle$.

  We now use the shuffle relation
  \begin{equation*}
    \lambda_2 \langle 2 P^{2m}\lambda_2, 2^2, [\ell]\rangle = 
    \langle \lambda_2, 2 P^{2m}\lambda_2, 2^2 \rangle [\ell].
  \end{equation*}
  Multiplication by $\lambda_2$ is an injection on the stems
  $\Pihat_{4j}(\bbQ)[\eta^{-1}] \to \Pihat_{4j+3}(\bbQ)[\eta^{-1}]$ by the
  product structure in the motivic Adams spectral sequence, hence the
  left hand side of the shuffle relation does not contain zero. As
  $[\ell]\cdot P^{2^{n-1}k}\lambda_n$ is in the right hand side of the
  shuffle relation, we conclude that $[\ell]\cdot
  P^{2^{n-1}k}\lambda_n$ is nonzero.

  A similar argument using the shuffle relation
  \begin{equation*}
    \lambda_2 \langle 2 P^{2k}v_2, 2^2, a_{\ell}\rangle = \langle \lambda_2, 2 P^{2k}v_2, 2^2\rangle a_{\ell}
  \end{equation*}
  establishes the claim that $a_{\ell}\cdot P^{2^{u-1}k}$ is non-zero.
\end{proof}

\begin{lemma}
  \label{lem:may_convergence}
  Let $m\geq 0$ and $\ell$ a prime, the Massey product $\langle \rho
  P^{2m}v_2, \rho^2, [\ell]\rangle$ in $\Ext(\bbQ)[h_1^{-1}]$ contains
  $[\ell]P^{2m+1}$ and has trivial indeterminacy.
\end{lemma}

\begin{proof}
  The $\rho$-Bockstein spectral sequence differential $d_3(P^{2m+1}) =
  \rho^3 P^{2m}v_2$ shows that $\langle \rho P^{2m}v_2, \rho^2,
  [\ell]\rangle$ contains $[\ell]P^{2m+1}$; this Massey product has
  trivial indeterminacy. To verify the hypotheses of May's convergence
  theorem \cite[Theorem 2.2.1]{StableStems}, first note the degree of
  $\rho^2[\ell]$ is $\epsilon=2$, $t=0$, and $c=1$. All
  $\rho$-Bockstein spectral sequence differentials vanish in this
  graded component. It remains to check $d_R$ differentials on the
  graded piece with $\epsilon' \geq 3$, $t = 8m +3$, $c=8m+4$ and $R >
  \epsilon'$ corresponding to $\rho^3 P^{2m}v_2$.

  We now look for elements of the $E_4$-page of the $\rho$-Bockstein
  spectral sequence which land in degrees $(t,c)$ for which $t+c
  \equiv 7 \bmod 8$. Given the description of the generators of the
  $E_4$-page, it suffices to consider products of just $v_n$ for
  $n\geq 3$, $[\ell]P$, $v_2$, and $\rho$. The sum $t+c \bmod 8$ for
  each of these generators is $0$, $1$, $4$, and $1$ respectively. As
  $([\ell]P)^2=0$, $\rho^3v_2 = 0$, and $\rho^2[\ell]Pv_2 = 0$ in the
  $E_4$-page, the only non-zero product of these generators in degree
  $(t,c)$ with $t+c \equiv 7 \bmod 8$ must be of the form
  $\rho^{\epsilon}v_3^{a_3}\cdots v_i^{a_i}$.

  Suppose now that $\rho^{\epsilon}v_3^{a_3}\cdots v_i^{a_i}$ is in
  degree $t=8m+3$, $c=8m+4$ for some $m$. Let $A = \sum a_i$ and $j =
  \min\{x \st a_x \neq 0 \}$. Under these assumptions it follows that
  \begin{equation*}
    8m + 3 = \sum a_i(2^i-1) \geq A(2^j -1).
  \end{equation*}
  Hence $\epsilon \geq A(2^j-2) + 1$. Note that $A$ must be at least 2
  in order for the Milnor-Witt stem $t$ to be congruent to 3 modulo
  8. It follows that if $R > \epsilon$, then $R > 2^j -1$ and so the
  class $\rho^{\epsilon}v_3^{a_3}\cdots v_i^{a_i}$ is zero in the
  $E_R$ page due to the relation $\rho^{2^j-1}v_j = 0$, which arises
  from a $d_{2^j-1}$ differential. We conclude May's convergence
  theorem applies in this situation.  It is straightforward to check
  that the indeterminacy is trivial.
\end{proof}

\bibliographystyle{alpha}
\bibliography{bibliography}

\begin{figure}
  \centering
\begin{tikzpicture}[x=.6cm, y=.6cm, scale=.8, transform shape]
\foreach \n in {0,...,24}{\draw[line width=.25pt, gray!40] (-0.5, \n-.5) -- (24, \n-.5);}
\foreach \n in {0,...,24}{\draw[line width=.25pt, gray!40] (\n-.5,-.5) -- (\n-.5,24);}
\foreach \n in {0,..., 24}{\draw (\n, -.6) node [anchor=north] {\n};}
\draw (-0.5, .2) node [anchor=east] {0};
\foreach \n in {1,...,24}{\draw (-0.5, \n) node [anchor=east] {\n};}


\draw [line width=.6pt] (0, 0)-- (0, 2);
\draw [line width=.6pt] (3, 1)-- (3, 3);
\draw [line width=.6pt] (6, 2)-- (6, 4);
\draw [line width=.6pt] (7, 1)-- (7, 2);
\draw [line width=.6pt] (7, 2)-- (7, 3);
\draw [line width=.6pt] (15, 1)-- (15, 2);
\draw [line width=.6pt] (15, 2)-- (15, 3);
\draw [line width=.6pt] (14, 2)-- (14, 3);
\draw [line width=.6pt] (14, 3)-- (14, 4);
\draw [line width=.6pt] (22, 2)-- (22, 3);
\draw [line width=.6pt] (22, 3)-- (22, 4);
\draw [line width=.6pt] (21, 3)-- (21, 4);
\draw [line width=.6pt] (21, 4)-- (21, 5);
\draw [line width=.6pt] (4, 4)-- (4, 5);
\draw [line width=.6pt] (4, 5)-- (4, 6);
\draw [line width=.6pt] (8, 8)-- (8, 9);
\draw [line width=.6pt] (8, 9)-- (8, 10);
\draw [line width=.6pt] (12, 12)-- (12, 13);
\draw [line width=.6pt] (12, 13)-- (12, 14);
\draw [line width=.6pt] (16, 16)-- (16, 17);
\draw [line width=.6pt] (16, 17)-- (16, 18);
\draw [line width=.6pt] (20, 20)-- (20, 21);
\draw [line width=.6pt] (20, 21)-- (20, 22);
\draw [line width=.6pt] (11, 5)-- (11, 6);
\draw [line width=.6pt] (11, 6)-- (11, 7);
\draw [line width=.6pt] (19, 5)-- (19, 6);
\draw [line width=.6pt] (19, 6)-- (19, 7);
\draw [line width=.6pt] (18, 6)-- (18, 7);
\draw [line width=.6pt] (18, 7)-- (18, 8);
\draw [line width=.6pt] (15, 9)-- (15, 10);
\draw [line width=.6pt] (15, 10)-- (15, 11);
\draw [line width=.6pt] (23, 9)-- (23, 10);
\draw [line width=.6pt] (23, 10)-- (23, 11);
\draw [line width=.6pt] (22, 10)-- (22, 11);
\draw [line width=.6pt] (22, 11)-- (22, 12);
\draw [line width=.6pt] (19, 13)-- (19, 14);
\draw [line width=.6pt] (19, 14)-- (19, 15);
\draw [line width=.6pt] (23, 17)-- (23, 18);
\draw [line width=.6pt] (23, 18)-- (23, 19);

\draw [line width=.6pt] (0, 0)-- (6, 2);
\draw [line width=.6pt] (0, 1)-- (6, 3);
\draw [line width=.6pt] (0, 2)-- (6, 4);
\draw [line width=.6pt, ->, >=stealth] (6, 2)-- (7, 2.33);
\draw [line width=.6pt, ->, >=stealth] (6, 3)-- (7, 3.33);
\draw [line width=.6pt, ->, >=stealth] (6, 4)-- (7, 4.33);
\draw [line width=.6pt, ->, >=stealth] (7, 1)-- (8, 1.33);
\draw [line width=.6pt, ->, >=stealth] (7, 2)-- (8, 2.33);
\draw [line width=.6pt, ->, >=stealth] (7, 3)-- (8, 3.33);
\draw [line width=.6pt, ->, >=stealth] (15, 1)-- (16, 1.33);
\draw [line width=.6pt, ->, >=stealth] (15, 2)-- (16, 2.33);
\draw [line width=.6pt, ->, >=stealth] (15, 3)-- (16, 3.33);
\draw [line width=.6pt, ->, >=stealth] (14, 2)-- (15, 2.33);
\draw [line width=.6pt, ->, >=stealth] (14, 3)-- (15, 3.33);
\draw [line width=.6pt, ->, >=stealth] (14, 4)-- (15, 4.33);
\draw [line width=.6pt, ->, >=stealth] (22, 2)-- (23, 2.33);
\draw [line width=.6pt, ->, >=stealth] (22, 3)-- (23, 3.33);
\draw [line width=.6pt, ->, >=stealth] (22, 4)-- (23, 4.33);
\draw [line width=.6pt, ->, >=stealth] (21, 3)-- (22, 3.33);
\draw [line width=.6pt, ->, >=stealth] (21, 4)-- (22, 4.33);
\draw [line width=.6pt, ->, >=stealth] (21, 5)-- (22, 5.33);
\draw [line width=.6pt, ->, >=stealth] (4, 4)-- (5, 4.33);
\draw [line width=.6pt, ->, >=stealth] (4, 5)-- (5, 5.33);
\draw [line width=.6pt, ->, >=stealth] (4, 6)-- (5, 6.33);
\draw [line width=.6pt, ->, >=stealth] (8, 8)-- (9, 8.33);
\draw [line width=.6pt, ->, >=stealth] (8, 9)-- (9, 9.33);
\draw [line width=.6pt, ->, >=stealth] (8, 10)-- (9, 10.33);
\draw [line width=.6pt, ->, >=stealth] (12, 12)-- (13, 12.33);
\draw [line width=.6pt, ->, >=stealth] (12, 13)-- (13, 13.33);
\draw [line width=.6pt, ->, >=stealth] (12, 14)-- (13, 14.33);
\draw [line width=.6pt, ->, >=stealth] (16, 16)-- (17, 16.33);
\draw [line width=.6pt, ->, >=stealth] (16, 17)-- (17, 17.33);
\draw [line width=.6pt, ->, >=stealth] (16, 18)-- (17, 18.33);
\draw [line width=.6pt, ->, >=stealth] (20, 20)-- (21, 20.33);
\draw [line width=.6pt, ->, >=stealth] (20, 21)-- (21, 21.33);
\draw [line width=.6pt, ->, >=stealth] (20, 22)-- (21, 22.33);
\draw [line width=.6pt, ->, >=stealth] (24, 24)-- (25, 24.33);
\draw [line width=.6pt, ->, >=stealth] (11, 5)-- (12, 5.33);
\draw [line width=.6pt, ->, >=stealth] (11, 6)-- (12, 6.33);
\draw [line width=.6pt, ->, >=stealth] (11, 7)-- (12, 7.33);
\draw [line width=.6pt, ->, >=stealth] (19, 5)-- (20, 5.33);
\draw [line width=.6pt, ->, >=stealth] (19, 6)-- (20, 6.33);
\draw [line width=.6pt, ->, >=stealth] (19, 7)-- (20, 7.33);
\draw [line width=.6pt, ->, >=stealth] (18, 6)-- (19, 6.33);
\draw [line width=.6pt, ->, >=stealth] (18, 7)-- (19, 7.33);
\draw [line width=.6pt, ->, >=stealth] (18, 8)-- (19, 8.33);
\draw [line width=.6pt, ->, >=stealth] (15, 9)-- (16, 9.33);
\draw [line width=.6pt, ->, >=stealth] (15, 10)-- (16, 10.33);
\draw [line width=.6pt, ->, >=stealth] (15, 11)-- (16, 11.33);
\draw [line width=.6pt, ->, >=stealth] (23, 9)-- (24, 9.33);
\draw [line width=.6pt, ->, >=stealth] (23, 10)-- (24, 10.33);
\draw [line width=.6pt, ->, >=stealth] (23, 11)-- (24, 11.33);
\draw [line width=.6pt, ->, >=stealth] (22, 10)-- (23, 10.33);
\draw [line width=.6pt, ->, >=stealth] (22, 11)-- (23, 11.33);
\draw [line width=.6pt, ->, >=stealth] (22, 12)-- (23, 12.33);
\draw [line width=.6pt, ->, >=stealth] (19, 13)-- (20, 13.33);
\draw [line width=.6pt, ->, >=stealth] (19, 14)-- (20, 14.33);
\draw [line width=.6pt, ->, >=stealth] (19, 15)-- (20, 15.33);
\draw [line width=.6pt, ->, >=stealth] (23, 17)-- (24, 17.33);
\draw [line width=.6pt, ->, >=stealth] (23, 18)-- (24, 18.33);
\draw [line width=.6pt, ->, >=stealth] (23, 19)-- (24, 19.33);
\filldraw [black, fill=black] (0, 0) circle (2pt);
\filldraw [black, fill=black] (3, 1) circle (2pt);
\filldraw [black, fill=black] (6, 2) circle (2pt);
\filldraw [black, fill=black] (7, 1) circle (2pt);
\filldraw [black, fill=black] (15, 1) circle (2pt);
\filldraw [black, fill=black] (14, 2) circle (2pt);
\filldraw [black, fill=black] (22, 2) circle (2pt);
\filldraw [black, fill=black] (21, 3) circle (2pt);
\filldraw [black, fill=black] (4, 4) circle (2pt);
\filldraw [black, fill=black] (8, 8) circle (2pt);
\filldraw [black, fill=black] (12, 12) circle (2pt);
\filldraw [black, fill=black] (16, 16) circle (2pt);
\filldraw [black, fill=black] (20, 20) circle (2pt);
\filldraw [black, fill=black] (24, 24) circle (2pt);
\filldraw [black, fill=black] (11, 5) circle (2pt);
\filldraw [black, fill=black] (19, 5) circle (2pt);
\filldraw [black, fill=black] (18, 6) circle (2pt);
\filldraw [black, fill=black] (15, 9) circle (2pt);
\filldraw [black, fill=black] (23, 9) circle (2pt);
\filldraw [black, fill=black] (22, 10) circle (2pt);
\filldraw [black, fill=black] (19, 13) circle (2pt);
\filldraw [black, fill=black] (23, 17) circle (2pt);
\filldraw [blue, fill=blue] (0, 1) circle (2pt);
\filldraw [blue, fill=blue] (3, 2) circle (2pt);
\filldraw [blue, fill=blue] (6, 3) circle (2pt);
\filldraw [blue, fill=blue] (7, 2) circle (2pt);
\filldraw [blue, fill=blue] (15, 2) circle (2pt);
\filldraw [blue, fill=blue] (14, 3) circle (2pt);
\filldraw [blue, fill=blue] (22, 3) circle (2pt);
\filldraw [blue, fill=blue] (21, 4) circle (2pt);
\filldraw [blue, fill=blue] (4, 5) circle (2pt);
\filldraw [blue, fill=blue] (8, 9) circle (2pt);
\filldraw [blue, fill=blue] (12, 13) circle (2pt);
\filldraw [blue, fill=blue] (16, 17) circle (2pt);
\filldraw [blue, fill=blue] (20, 21) circle (2pt);
\filldraw [blue, fill=blue] (11, 6) circle (2pt);
\filldraw [blue, fill=blue] (19, 6) circle (2pt);
\filldraw [blue, fill=blue] (18, 7) circle (2pt);
\filldraw [blue, fill=blue] (15, 10) circle (2pt);
\filldraw [blue, fill=blue] (23, 10) circle (2pt);
\filldraw [blue, fill=blue] (22, 11) circle (2pt);
\filldraw [blue, fill=blue] (19, 14) circle (2pt);
\filldraw [blue, fill=blue] (23, 18) circle (2pt);
\filldraw [red, fill=red] (0, 2) circle (2pt);
\filldraw [red, fill=red] (3, 3) circle (2pt);
\filldraw [red, fill=red] (6, 4) circle (2pt);
\filldraw [red, fill=red] (7, 3) circle (2pt);
\filldraw [red, fill=red] (15, 3) circle (2pt);
\filldraw [red, fill=red] (14, 4) circle (2pt);
\filldraw [red, fill=red] (22, 4) circle (2pt);
\filldraw [red, fill=red] (21, 5) circle (2pt);
\filldraw [red, fill=red] (4, 6) circle (2pt);
\filldraw [red, fill=red] (8, 10) circle (2pt);
\filldraw [red, fill=red] (12, 14) circle (2pt);
\filldraw [red, fill=red] (16, 18) circle (2pt);
\filldraw [red, fill=red] (20, 22) circle (2pt);
\filldraw [red, fill=red] (11, 7) circle (2pt);
\filldraw [red, fill=red] (19, 7) circle (2pt);
\filldraw [red, fill=red] (18, 8) circle (2pt);
\filldraw [red, fill=red] (15, 11) circle (2pt);
\filldraw [red, fill=red] (23, 11) circle (2pt);
\filldraw [red, fill=red] (22, 12) circle (2pt);
\filldraw [red, fill=red] (19, 15) circle (2pt);
\filldraw [red, fill=red] (23, 19) circle (2pt);

\draw (0, 0) node [anchor=east] {$1$};
\draw (3, 1) node [anchor=east] {$v_2$};
\draw (7, 1) node [anchor=east] {$v_3$};
\draw (15, 1) node [anchor=east] {$v_4$};
\draw (14, 2) node [anchor=east] {$v_3^2$};
\draw (22, 2) node [anchor=east] {$v_3v_4$};
\draw (21, 3) node [anchor=east] {$v_3^3$};
\draw (4, 4) node [anchor=east] {$P$};
\draw (8, 8) node [anchor=east] {$P^2$};
\draw (12, 12) node [anchor=east] {$P^3$};
\draw (16, 16) node [anchor=east] {$P^4$};
\draw (20, 20) node [anchor=east] {$P^5$};
\draw (24, 24) node [anchor=east] {$P^6$};
\draw (11, 5) node [anchor=east] {$Pv_3$};
\draw (19, 5) node [anchor=east] {$Pv_4$};
\draw (18, 6) node [anchor=east] {$Pv_3^2$};
\draw (15, 9) node [anchor=east] {$P^2v_3$};
\draw (23, 9) node [anchor=east] {$P^2v_4$};
\draw (22, 10) node [anchor=east] {$P^2v_3^2$};
\draw (19, 13) node [anchor=east] {$P^3v_3$};
\draw (23, 17) node [anchor=east] {$P^4v_3$};

\end{tikzpicture}
\caption{ The $E_1$ page of the $h_1$-inverted $\rho$-Bockstein
  spectral sequence over a field $F$ with $\mathrm{cd}_2(F)=2$ up to
  Milnor-Witt stem $24$. Solid vertical lines indicate possible
  $\rho$-multiplications that depend on the field. Black dots
  represent the group $\bbZ/2[h_1^{\pm 1}]$, blue dots represent
  $k^M_1(F)[h_1^{\pm 1}]$, and red dots represent $k^M_2(F)[h_1^{\pm
    1}]$. Solid lines of slope $1/3$ indicate multiplication by $v_2$
  and arrows in this direction represent a tower of nonzero $v_2$
  multiples. The horizontal axis $t$ is the Milnor-Witt stem and the
  vertical axis $c$ is the Chow weight, while the Adams filtration is
  suppressed.}
\label{fig:cd2BSS-E1}

\end{figure}
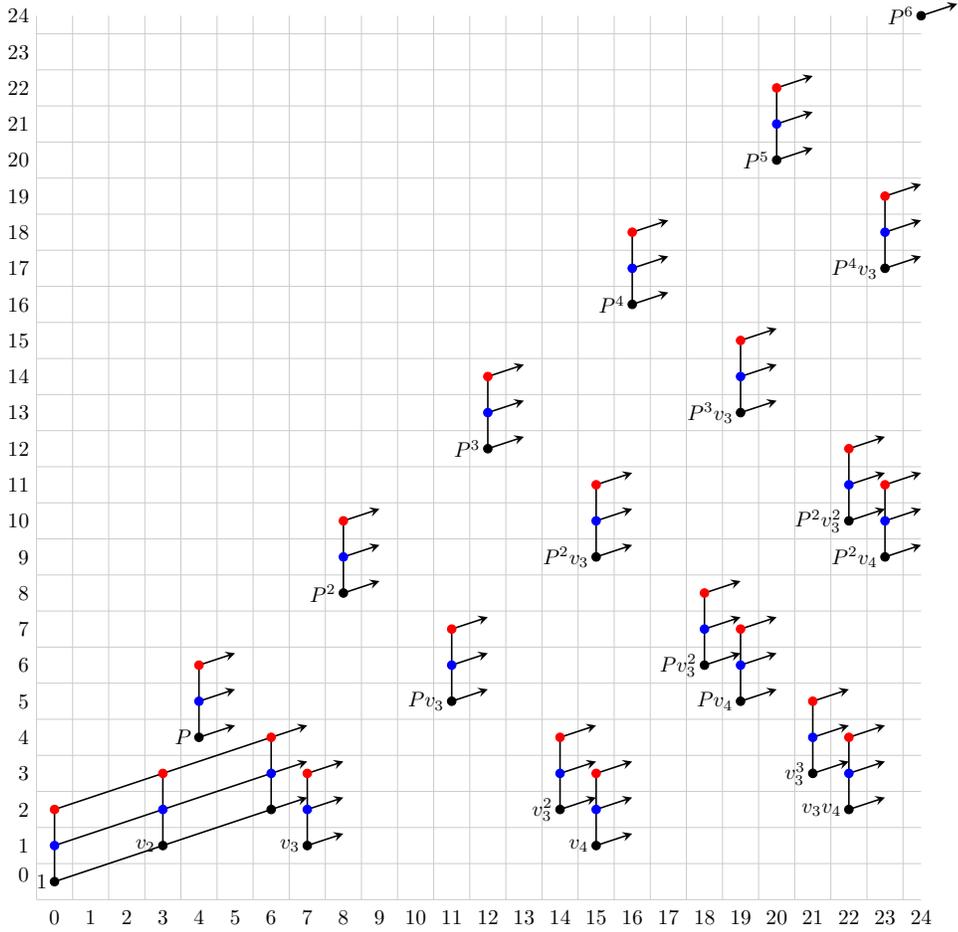

\begin{figure}
  \centering
\begin{tikzpicture}[x=.6cm, y=.6cm, scale=.8, transform shape]
\foreach \n in {0,...,24}{\draw[line width=.25pt, gray!40] (-0.5, \n-.5) -- (24, \n-.5);}
\foreach \n in {0,...,24}{\draw[line width=.25pt, gray!40] (\n-.5,-.5) -- (\n-.5,24);}
\foreach \n in {0,..., 24}{\draw (\n, -.6) node [anchor=north] {\n};}
\draw (-0.5, .2) node [anchor=east] {0};
\foreach \n in {1,...,24}{\draw (-0.5, \n) node [anchor=east] {\n};}


\draw [line width=.6pt] (0, 0)-- (0, 1);
\draw [line width=.6pt] (0, 1)-- (0, 2);

\draw [line width=.6pt] (3, 1)-- (3, 2);
\draw [line width=.6pt] (3, 2)-- (3, 3);

\draw [line width=.6pt] (4, 4)-- (4, 5);
\draw [line width=.6pt] (4, 5)-- (4, 6);

\draw [line width=.6pt] (7, 5)-- (7, 6);
\draw [line width=.6pt] (7, 6)-- (7, 7);

\draw [line width=.6pt] (8, 8)-- (8, 9);
\draw [line width=.6pt] (8, 9)-- (8, 10);

\draw [line width=.6pt] (11, 9)-- (11, 10);
\draw [line width=.6pt] (11, 10)-- (11, 11);

\draw [line width=.6pt] (12, 12)-- (12, 13);
\draw [line width=.6pt] (12, 13)-- (12, 14);

\draw [line width=.6pt] (15, 13)-- (15, 14);
\draw [line width=.6pt] (15, 14)-- (15, 15);

\draw [line width=.6pt] (16, 16)-- (16, 17);
\draw [line width=.6pt] (16, 17)-- (16, 18);

\draw [line width=.6pt] (19, 17)-- (19, 18);
\draw [line width=.6pt] (19, 18)-- (19, 19);

\draw [line width=.6pt] (20, 20)-- (20, 21);
\draw [line width=.6pt] (20, 21)-- (20, 22);

\draw [line width=.6pt] (23, 21)-- (23, 22);
\draw [line width=.6pt] (23, 22)-- (23, 23);

\draw [line width=.6pt] (20, 20)-- (20, 21);
\draw [line width=.6pt] (20, 21)-- (20, 22);

\draw [line width=.6pt] (23, 21)-- (23, 22);
\draw [line width=.6pt] (23, 22)-- (23, 23);

\foreach \n in {0,1,2}{\draw [line width=.2pt] (0, 0+\n)-- (3, 1+\n);}
\foreach \n in {0,1,2}{\draw [line width=.2pt] (4, 4+\n)-- (3+4, 4+1+\n);}
\foreach \n in {0,1,2}{\draw [line width=.2pt] (8, 8+\n)-- (8+3, 8+1+\n);}
\foreach \n in {0,1,2}{\draw [line width=.2pt] (12, 12+\n)-- (12+3, 12+1+\n);}
\foreach \n in {0,1,2}{\draw [line width=.2pt] (16, 16+\n)-- (16+3, 16+1+\n);}
\foreach \n in {0,1,2}{\draw [line width=.2pt] (20, 20+\n)-- (20+3, 20+1+\n);}

\filldraw [black, fill=black] (0, 0) circle (2pt);
\filldraw [black, fill=black] (3, 1) circle (2pt);
\filldraw [black, fill=black] (4, 4) circle (2pt);
\filldraw [black, fill=black] (8, 8) circle (2pt);
\filldraw [black, fill=black] (12, 12) circle (2pt);
\filldraw [black, fill=black] (16, 16) circle (2pt);
\filldraw [black, fill=black] (20, 20) circle (2pt);
\filldraw [black, fill=black] (24, 24) circle (2pt);
\filldraw [blue, fill=blue] (0, 1) circle (2pt);
\filldraw [blue, fill=blue] (3, 2) circle (2pt);
\filldraw [blue, fill=blue] (4, 5) circle (2pt);
\filldraw [blue, fill=blue] (8, 9) circle (2pt);
\filldraw [blue, fill=blue] (12, 13) circle (2pt);
\filldraw [blue, fill=blue] (16, 17) circle (2pt);
\filldraw [blue, fill=blue] (20, 21) circle (2pt);
\filldraw [red, fill=red] (0, 2) circle (2pt);
\filldraw [red, fill=red] (3, 3) circle (2pt);
\filldraw [red, fill=red] (4, 6) circle (2pt);
\filldraw [red, fill=red] (8, 10) circle (2pt);
\filldraw [red, fill=red] (12, 14) circle (2pt);
\filldraw [red, fill=red] (16, 18) circle (2pt);
\filldraw [red, fill=red] (20, 22) circle (2pt);

\filldraw [black, fill=black] (7, 5) circle (2pt);
\filldraw [blue, fill=blue] (7, 6) circle (2pt);
\filldraw [red, fill=red] (7, 7) circle (2pt);

\filldraw [black, fill=black] (11, 9) circle (2pt);
\filldraw [blue, fill=blue] (11, 10) circle (2pt);
\filldraw [red, fill=red] (11, 11) circle (2pt);

\filldraw [black, fill=black] (15, 13) circle (2pt);
\filldraw [blue, fill=blue] (15, 14) circle (2pt);
\filldraw [red, fill=red] (15, 15) circle (2pt);

\filldraw [black, fill=black] (19, 17) circle (2pt);
\filldraw [blue, fill=blue] (19, 18) circle (2pt);
\filldraw [red, fill=red] (19, 19) circle (2pt);

\filldraw [black, fill=black] (23, 21) circle (2pt);
\filldraw [blue, fill=blue] (23, 22) circle (2pt);
\filldraw [red, fill=red] (23, 23) circle (2pt);

\draw (0, 0) node [anchor=east] {$1$};
\draw (3, 1) node [anchor=east] {$v_2$};
\draw (4, 4) node [anchor=east] {$P$};
\draw (8, 8) node [anchor=east] {$P^2$};
\draw (12, 12) node [anchor=east] {$P^3$};
\draw (16, 16) node [anchor=east] {$P^4$};
\draw (20, 20) node [anchor=east] {$P^5$};
\draw (24, 24) node [anchor=east] {$P^6$};

\end{tikzpicture}
\caption{The $E_{\infty}$ page of the $h_1$-inverted motivic Adams
  spectral sequence over a field $F$ with $\cd_2(F)\leq 2$ up to
  Milnor-Witt stem $24$.  
  The notational conventions of figure \ref{fig:cd2BSS-E1} apply here.}
\label{fig:cd2BSS-Einfty}
\end{figure}
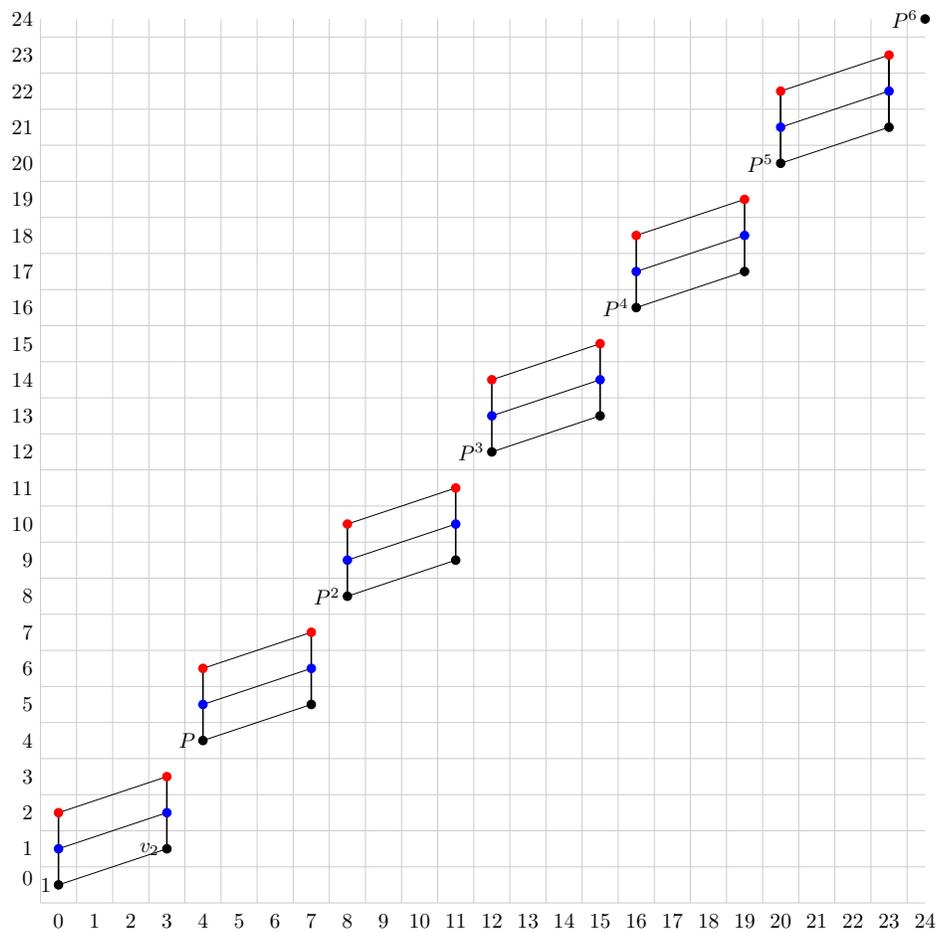

\begin{figure}
  \centering
\begin{tikzpicture}[x=.85cm, y=.6cm, scale=.9, transform shape]
\foreach \n in {0,...,19}{\draw[line width=.25pt, gray!40] (-0.5, \n-.5) -- (15.5, \n-.5);}
\foreach \n in {0,...,16}{\draw[line width=.25pt, gray!40] (\n-.5,-.5) -- (\n-.5,19);}
\foreach \n in {0,..., 15}{\draw (\n, -.6) node [anchor=north] {\n};}
\draw (-0.5, .2) node [anchor=east] {0};
\foreach \n in {1,...,19}{\draw (-0.5, \n) node [anchor=east] {\n};}


\draw [line width=.6pt, ->] (0+.3, 0)-- (0+.3, 19);
\draw [line width=.6pt, ->] (3+.3, 1)-- (3+.3, 19);
\draw [line width=.6pt, ->] (7+.3, 1)-- (7+.3, 19);
\draw [line width=.6pt, ->] (15+.2, 1)-- (15+.2, 19);
\draw [line width=.6pt, ->] (14+.3, 2)-- (14+.3, 19);
\draw [line width=.6pt, ->] (4+.3, 4)-- (4+.3, 19);
\draw [line width=.6pt, ->] (8+.3, 8)-- (8+.3, 19);
\draw [line width=.6pt, ->] (12+.3, 12)-- (12+.3, 19);
\draw [line width=.6pt, ->] (11+.3, 5)-- (11+.3, 19);
\draw [line width=.6pt, ->] (15-.2, 9)-- (15-.2, 19);

\foreach \n in {1,2}{\draw [line width=.6pt] (0, 0+\n+.1)-- (3, 1+\n+.1);}
\foreach \n in {1,2}{\draw [line width=.6pt] (0-.3, 0+\n+.2)-- (3-.3, 1+\n+.2);}
\foreach \n in {0,1}{\draw [line width=.6pt, ->, >=stealth] (3, 2+.1+\n)-- (4, 2.33+.1+\n);}
\foreach \n in {0,1}{\draw [line width=.6pt, ->, >=stealth] (3-.3, 2+.2+\n)-- (4-.3, 2.33+.2+\n);}

\foreach \n in {0}{\draw [line width=.6pt] (0+.3, \n)-- (3+.3, 1+\n);}
\foreach \n in {0}{\draw [line width=.6pt, ->, >=stealth] (3+.3, 1+\n)-- (4+.3, 1.33+\n);}
\foreach \n in {0}{\draw [line width=.6pt, ->, >=stealth] (7+.3, 1+\n)-- (8+.3, 1.33+\n);}
\foreach \n in {0}{\draw [line width=.6pt, ->, >=stealth] (15+.2, 1+\n)-- (16+.2, 1.33+\n);}
\foreach \n in {0}{\draw [line width=.6pt, ->, >=stealth] (14+.3, 2+\n)-- (15+.3, 2.33+\n);}
\foreach \n in {0}{\draw [line width=.6pt, ->, >=stealth] (4+.3, 4+\n)-- (5+.3, 4.33+\n);}
\foreach \n in {0}{\draw [line width=.6pt, ->, >=stealth] (8+.3, 8+\n)-- (9+.3, 8.33+\n);}
\foreach \n in {0}{\draw [line width=.6pt, ->, >=stealth] (12+.3, 12+\n)-- (13+.3, 12.33+\n);}

\foreach \n in {0}{\draw [line width=.6pt, ->, >=stealth] (11+.3, 5+\n)-- (12+.3, 5.33+\n);}
\foreach \n in {0}{\draw [line width=.6pt, ->, >=stealth] (15-.2, 9+\n)-- (16-.2, 9.33+\n);}

\foreach \n in {0,...,18}{\filldraw [black, fill=black] (0+.3, \n) circle (2pt);}
\foreach \n in {0,...,17}{\filldraw [black, fill=black] (3+.3, 1+\n) circle (2pt);}
\foreach \n in {0,...,17}{\filldraw [black, fill=black] (7+.3, 1+\n) circle (2pt);}
\foreach \n in {0,...,17}{\filldraw [black, fill=black] (15+.2, 1+\n) circle (2pt);}
\foreach \n in {0,...,16}{\filldraw [black, fill=black] (14+.3, 2+\n) circle (2pt);}
\foreach \n in {0,...,14}{\filldraw [black, fill=black] (4+.3, 4+\n) circle (2pt);}
\foreach \n in {0,...,10}{\filldraw [black, fill=black] (8+.3, 8+\n) circle (2pt);}
\foreach \n in {0,...,6}{\filldraw [black, fill=black] (12+.3, 12+\n) circle (2pt);}
\foreach \n in {0,...,13}{\filldraw [black, fill=black] (11+.3, 5+\n) circle (2pt);}
\foreach \n in {0,...,9}{\filldraw [black, fill=black] (15-.2, 9+\n) circle (2pt);}

\onetwomid{0}{1}
\onetwomid{4}{5}
\onetwomid{8}{9}
\onetwomid{12}{13}

\onetwomid{3}{2}
\onetwomid{3}{2}
\onetwomid{11}{6}

\onetwomid{7}{2}
\onetwomid{12}{13}
\onetwomid{14}{3}

\onetwomid{15}{2}

\draw[line width=.6pt] (15-.05, 10+.3) -- (15-.05, 10+1+.4);
\foreach \n in {0,1}{\filldraw [black, fill=blue] (15-.05, \n+10+.3) circle (2pt);}
\foreach \n in {0}{\filldraw [black, fill=red] (15-.35, \n+10+.3) circle (2pt);} 
\foreach \n in {1}{\filldraw [black, fill=green] (15-.35, \n+10+.3) circle (2pt);}

\draw (0, 0) node [anchor=north] {$1$};
\draw (3.3, 1) node [anchor=north] {$v_2$};
\draw (7.3, 1) node [anchor=north] {$v_3$};
\draw (15.3, 1) node [anchor=north] {$v_4$};
\draw (14.3, 2) node [anchor=north] {$v_3^2$};
\draw (4.3, 4) node [anchor=north] {$P$};
\draw (8.3, 8) node [anchor=north] {$P^2$};
\draw (12.3, 12) node [anchor=north] {$P^3$};
\draw (11.3, 5) node [anchor=north] {$Pv_3$};
\draw (15-.25, 9) node [anchor=north] {$P^2v_3$};

\end{tikzpicture}
\caption{The $E_1$ page of the $h_1$-inverted $\rho$-Bockstein
  spectral sequence over $\bbQ$ up to Milnor-Witt stem $15$.  Black
  dots represent the group $\bbZ/2[h_1^{\pm 1}]$, blue dots represent
  $\oplus_{p\equiv 3\, (4)} \bbZ/2[h_1^{\pm 1}]$, red dots represent
  $\oplus_{p\equiv 1\,(4), p=2} \bbZ/2[h_1^{\pm 1}]$, and green dots
  represent $\oplus_{p\equiv 1\,(4)} \bbZ/2[h_1^{\pm 1}]$. Solid
  vertical lines indicate multiplication by $\rho$ and a vertical
  arrow means that the tower of $\rho$-multiplications continues
  indefinitely. Every dot supports an infinite tower of
  $v_2$-multiples, however we only indicate this with lines and arrows
  of slope $1/3$ on the classes of $\Ext(\bbC)[h_1^{-1}]$ and
  $k^M_*(\bbQ)$.  The horizontal axis $t$ is the Milnor-Witt stem and
  the vertical axis $c$ is the Chow weight, while the Adams filtration
  is suppressed.}
\label{fig:Q-RhoBSS}
\end{figure}

\begin{sidewaysfigure}
  \centering
\begin{tikzpicture}[x=1cm, y=.63cm, scale=1, transform shape]
\foreach \n in {0,...,16}{\draw[line width=.25pt, gray!40] (-0.5, \n-.5) -- (15.5, \n-.5);}
\foreach \n in {0,...,16}{\draw[line width=.25pt, gray!40] (\n-.5,-.5) -- (\n-.5,16);}
\foreach \n in {0,..., 15}{\draw (\n, -.6) node [anchor=north] {\n};}
\draw (-0.5, .2) node [anchor=east] {0};
\foreach \n in {1,...,15}{\draw (-0.5, \n) node [anchor=east] {\n};}

\draw[line width=.6pt, ->] (0+.3, 0) -- (0+.3, 16);

\foreach \n in {0,1}{\draw[line width=.6pt] (3+.3, \n+1) -- (3+.3, \n+2);}
\foreach \n in {0,1}{\draw[line width=.6pt] (6+.3, \n+2) -- (6+.3, \n+3);}
\foreach \n in {0,1}{\draw[line width=.6pt] (9+.3, \n+3) -- (9+.3, \n+4);}
\foreach \n in {0,1}{\draw[line width=.6pt] (10+.3, \n+2) -- (10+.3, \n+3);}
\foreach \n in {0,1}{\draw[line width=.6pt] (11+.3, \n+9) -- (11+.3, \n+10);}
\foreach \n in {0,1}{\draw[line width=.6pt] (14+.3, \n+10) -- (14+.3, \n+11);}

\foreach \n in {0,...,5}{\draw[line width=.6pt] (7+.3, \n+1) -- (7+.3, \n+2);}
\foreach \n in {0,...,5}{\draw[line width=.6pt] (14+.3, \n+2) -- (14+.3, \n+3);}

\foreach \n in {0,...,13}{\draw[line width=.6pt] (15+.3, \n+1) -- (15+.3, \n+2);}

\foreach \n in {0,1,2}{\draw [line width=.6pt] (0+.3, 0+\n)-- (3+.3, 1+\n);}
\foreach \n in {0,1,2}{\draw [line width=.6pt] (3+.3, 1+\n)-- (6+.3, 2+\n);}
\foreach \n in {0,1,2}{\draw [line width=.6pt] (6+.3, 2+\n)-- (9+.3, 3+\n);}
\foreach \n in {0,1,2}{\draw [line width=.6pt, ->, >=stealth] (9+.3, 3+\n)-- (10.5+.3, 3.49+\n);}

\foreach \n in {0,1,2}{\draw [line width=.6pt] (11+.3, 9+\n)-- (14+.3, 10+\n);}

\foreach \n in {0,1,2}{\draw [line width=.6pt] (7+.3, 1+\n)-- (10+.3, 2+\n);}
\foreach \n in {0,1,2}{\draw [line width=.6pt, ->, >=stealth] (10+.3, 2+\n)-- (11.5+.3, 2.49+\n);}
\foreach \n in {0,1,2}{\draw [line width=.6pt, ->, >=stealth] (14+.3, 2+\n)-- (15.5+.3, 2.49+\n);}
\foreach \n in {0,1,2}{\draw [line width=.6pt, ->, >=stealth] (14+.3, 10+\n)-- (15.5+.3, 10.49+\n);}

\foreach \n in {0,1,2}{\draw [line width=.6pt, red] (7+.3, 1+\n)-- (6+.3, 2+\n);}
\foreach \n in {0,1,2}{\draw [line width=.6pt, red] (10+.3, 2+\n)-- (9+.3, 3+\n);}
\foreach \n in {0,...,6}{\draw [line width=.6pt, red] (15+.3, 1+\n)-- (14+.3, 2+\n);}

\foreach \n in {0,...,15}{\filldraw [black, fill=black] (0+.3, \n) circle (2pt);}

\foreach \n in {0,1,2}{\filldraw [black, fill=black] (3+.3, \n+1) circle (2pt);}
\foreach \n in {0,1,2}{\filldraw [black, fill=black] (6+.3, \n+2) circle (2pt);}
\foreach \n in {0,1,2}{\filldraw [black, fill=black] (9+.3, \n+3) circle (2pt);}
\foreach \n in {0,1,2}{\filldraw [black, fill=black] (10+.3, \n+2) circle (2pt);}
\foreach \n in {0,1,2}{\filldraw [black, fill=black] (11+.3, \n+9) circle (2pt);}
\foreach \n in {0,1,2}{\filldraw [black, fill=black] (14+.3, \n+10) circle (2pt);}

\foreach \n in {0,...,6}{\filldraw [black, fill=black] (7+.3, \n+1) circle (2pt);}
\foreach \n in {0,...,6}{\filldraw [black, fill=black] (14+.3, \n+2) circle (2pt);}

\foreach \n in {0,...,14}{\filldraw [black, fill=black] (15+.3, \n+1) circle (2pt);}

\foreach \n in {0,...,3}{\filldraw [black, fill=black] (7+.3, \n+4) circle (2pt);}
\foreach \n in {0,1,2}{\filldraw [black, fill=black] (11+.3, \n+9) circle (2pt);}
\foreach \n in {0,...,4}{\filldraw [black, fill=black] (15+.3, \n+11) circle (2pt);}

\newdtwomed{7}{2}
\newdtwomed{10}{3}
\newdtwomed{11}{6}
\newdtwomed{15}{2}
\newdtwomed{15}{10}

\foreach \n in {1,2}{\draw [line width=.6pt] (0, 0+.1+\n)-- (9, 3+.1+\n);}
\foreach \n in {1,2}{\draw [line width=.6pt] (4, 4+.1+\n)-- (10, 6+.1+\n);}
\foreach \n in {1,2}{\draw [line width=.6pt] (8, 8+.1+\n)-- (14, 10+.1+\n);}
\foreach \n in {1,2}{\draw [line width=.6pt] (12, 12+.1+\n)-- (15, 13+.1+\n);}
\foreach \n in {1,2}{\draw [line width=.6pt] (0-.3, 0+.2+\n)-- (9-.3, 3+.2+\n);}
\foreach \n in {1,2}{\draw [line width=.6pt] (4-.3, 4+.2+\n)-- (10-.3, 6+.2+\n);}
\foreach \n in {1,2}{\draw [line width=.6pt] (8-.3, 8+.2+\n)-- (14-.3, 10+.2+\n);}
\foreach \n in {1,2}{\draw [line width=.6pt] (12-.3, 12+.2+\n)-- (15-.3, 13+.2+\n);}

\vtwomultmid{9}{4}
\vtwomultmid{10}{3}
\vtwomultmid{10}{7}
\vtwomultmid{11}{6}
\vtwomultmid{14}{11}
\vtwomultmid{14}{3}
\vtwomultmid{15}{2}
\vtwomultmid{15}{10}
\vtwomultmid{15}{14}

\onetwomid{0}{1}
\onetwomid{4}{5}
\onetwomid{8}{9}
\onetwomid{12}{13}
\onetwomid{3}{2}
\onetwomid{7}{6}
\onetwomid{11}{10}
\onetwomid{15}{14}

\onetwomid{3}{2}
\onetwomid{6}{3}
\onetwomid{9}{4}

\onetwomid{11}{6}

\onetwomid{7}{2}
\onetwomid{10}{3}
\onetwomid{10}{7}
\onetwomid{12}{13}
\onetwomid{14}{3}
\onetwomid{14}{11}
\onetwomid{15}{2}
\onetwomid{15}{10}

\draw (3+.3, 1) node [anchor=north] {$v_2$};
\draw (7+.3, 1) node [anchor=north] {$v_3$};
\draw (15+.3, 1) node [anchor=north] {$v_4$};
\draw (14+.3, 2) node [anchor=north] {$v_3^2$};
\draw (11+.3, 9) node [anchor=north] {$P^2v_2$};

\end{tikzpicture}
\caption{The $E_{\infty}$ page of $\BSS(\bbQ)[h_1^{-1}]$, which is
  also $\Ext(\bbQ)[h_1^{-1}]$, up to Milnor-Witt stem 15.  Black dots
  represent the group $\bbZ/2[h_1^{\pm 1}]$, blue dots represent
  $\oplus_{p\equiv 3\, (4)} \bbZ/2[h_1^{\pm 1}]$, red dots represent
  $\oplus_{p\equiv 1\,(4), p=2} \bbZ/2[h_1^{\pm 1}]$, and green dots
  represent $\oplus_{p\equiv 1\,(4)} \bbZ/2[h_1^{\pm 1}]$. Solid
  vertical lines indicate multiplication by $\rho$ and a vertical
  arrow means that the tower of $\rho$-multiplications continues
  indefinitely. Multiplication by $v_2$ is indicated by lines of slope
  $1/3$ and an arrow of slope $1/3$ indicates that the class supports
  a tower of $v_2$-multiplications.  The $d_2$ differentials of
  $\mathfrak{M}(\bbQ)[h_1^{-1}]$ are indicated with red lines of slope
  -1. }
\label{fig:Q-RhoBSSInfinity}
\end{sidewaysfigure}

\begin{figure}
  \centering
\begin{tikzpicture}[x=.77cm, y=.77cm, scale=1, transform shape]
\foreach \n in {0,...,16}{\draw[line width=.25pt, gray!40] (-0.5, \n-.5) -- (15.5, \n-.5);}
\foreach \n in {0,...,16}{\draw[line width=.25pt, gray!40] (\n-.5,-.5) -- (\n-.5,15.5);}
\foreach \n in {0,..., 15}{\draw (\n, -.6) node [anchor=north] {\n};}
\draw (-0.5, .2) node [anchor=east] {0};
\foreach \n in {1,...,15}{\draw (-0.5, \n) node [anchor=east] {\n};}

\draw[line width=.6pt, ->] (0+.3, 0) -- (0+.3, 15);

\foreach \n in {0,1}{\draw[line width=.6pt] (3+.3, \n+1) -- (3+.3, \n+2);}
\foreach \n in {0,1}{\draw[line width=.6pt] (11+.3, \n+9) -- (11+.3, \n+10);}
\foreach \n in {0,1}{\draw[line width=.6pt] (14+.3, \n+10) -- (14+.3, \n+11);}

\draw[line width=.6pt] (7+.3, 4) -- (7+.3, 7);

\draw[line width=.6pt] (15+.3, 8) -- (15+.3, 15);

\foreach \n in {0,1,2}{\draw [line width=.6pt] (0+.3, 0+\n)-- (3+.3, 1+\n);}
\foreach \n in {0,1,2}{\draw [line width=.6pt] (11+.3, 9+\n)-- (14+.3, 10+\n);}

\foreach \n in {0,1,2}{\draw [line width=.6pt, blue] (15+.3, 8+\n)-- (14+.3, 10+\n);}

\foreach \n in {0,...,14}{\filldraw [black, fill=black] (0+.3, \n) circle (2pt);}

\foreach \n in {0,1,2}{\filldraw [black, fill=black] (3+.3, \n+1) circle (2pt);}
\foreach \n in {0,...,3}{\filldraw [black, fill=black] (7+.3, \n+4) circle (2pt);}
\foreach \n in {0,1,2}{\filldraw [black, fill=black] (11+.3, \n+9) circle (2pt);}
\foreach \n in {0,1,2}{\filldraw [black, fill=black] (14+.3, \n+10) circle (2pt);}
\foreach \n in {7,...,14}{\filldraw [black, fill=black] (15+.3, \n+1) circle (2pt);}
\foreach \n in {0,...,4}{\filldraw [black, fill=black] (15+.3, \n+11) circle (2pt);}

\foreach \n in {1,2}{\draw [line width=.6pt] (0, 0+.1+\n)-- (3, 1+.1+\n);}
\foreach \n in {1,2}{\draw [line width=.6pt] (4, 4+.1+\n)-- (7, 5+.1+\n);}
\foreach \n in {1,2}{\draw [line width=.6pt] (8, 8+.1+\n)-- (11, 9+.1+\n);}
\foreach \n in {1,2}{\draw [line width=.6pt] (12, 12+.1+\n)-- (15, 13+.1+\n);}
\foreach \n in {1,2}{\draw [line width=.6pt] (0-.3, 0+.2+\n)-- (3-.3, 1+.2+\n);}
\foreach \n in {1,2}{\draw [line width=.6pt] (4-.3, 4+.2+\n)-- (7-.3, 5+.2+\n);}
\foreach \n in {1,2}{\draw [line width=.6pt] (8-.3, 8+.2+\n)-- (11-.3, 9+.2+\n);}
\foreach \n in {1,2}{\draw [line width=.6pt] (12-.3, 12+.2+\n)-- (15-.3, 13+.2+\n);}

\onetwomid{0}{1}
\onetwomid{4}{5}
\onetwomid{8}{9}
\onetwomid{12}{13}
\onetwomid{3}{2}
\onetwomid{7}{6}
\onetwomid{11}{10}
\onetwomid{15}{14}

\draw (3+.3, 1) node [anchor=north] {$v_2$};
\draw (7+.3, 4) node [anchor=north] {$\rho^3v_3$};
\draw (11+.3, 9) node [anchor=north] {$P^2v_2$};
\draw (15+.3, 8) node [anchor=north] {$\rho^8v_4$};

\end{tikzpicture}
\caption{The $E_3$ page of $\MASS(\bbQ)[h_1^{-1}]$ up to Milnor-Witt
  stem 15. The $d_3$ differentials are indicated with blue lines of
  slope -2. The notational conventions of figure
  \ref{fig:Q-RhoBSSInfinity} apply here.}
\label{fig:Q-E3}
\end{figure}
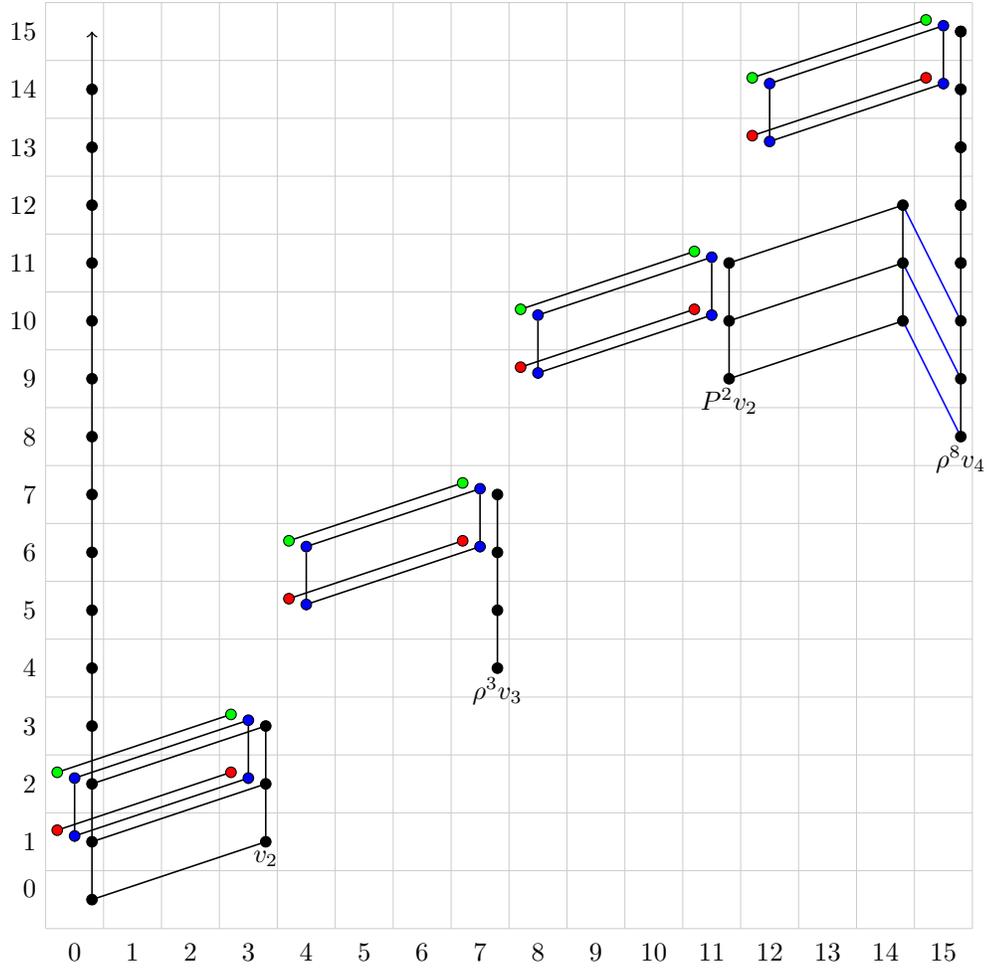

\begin{figure}
  \centering
\begin{tikzpicture}[x=.77cm, y=.8cm, scale=1, transform shape]
\foreach \n in {0,...,16}{\draw[line width=.25pt, gray!40] (-0.5, \n-.5) -- (15.5, \n-.5);}
\foreach \n in {0,...,16}{\draw[line width=.25pt, gray!40] (\n-.5,-.5) -- (\n-.5,15.5);}
\foreach \n in {0,..., 15}{\draw (\n, -.6) node [anchor=north] {\n};}
\draw (-0.5, .2) node [anchor=east] {0};
\foreach \n in {1,...,15}{\draw (-0.5, \n) node [anchor=east] {\n};}

\draw[line width=.6pt, ->] (0+.3, 0) -- (0+.3, 15);

\foreach \n in {0,1}{\draw[line width=.6pt] (3+.3, \n+1) -- (3+.3, \n+2);}
\foreach \n in {0,1}{\draw[line width=.6pt] (11+.3, \n+9) -- (11+.3, \n+10);}

\draw[line width=.6pt] (7+.3, 4) -- (7+.3, 7);

\draw[line width=.6pt] (15+.3, 11) -- (15+.3, 15);

\foreach \n in {0,1,2}{\draw [line width=.6pt] (0+.3, 0+\n)-- (3+.3, 1+\n);}

\foreach \n in {0,...,14}{\filldraw [black, fill=black] (0+.3, \n) circle (2pt);}

\foreach \n in {0,1,2}{\filldraw [black, fill=black] (3+.3, \n+1) circle (2pt);}
\foreach \n in {0,...,3}{\filldraw [black, fill=black] (7+.3, \n+4) circle (2pt);}
\foreach \n in {0,1,2}{\filldraw [black, fill=black] (11+.3, \n+9) circle (2pt);}
\foreach \n in {10,...,14}{\filldraw [black, fill=black] (15+.3, \n+1) circle (2pt);}

\foreach \n in {1,2}{\draw [line width=.6pt] (0, 0+.1+\n)-- (3, 1+.1+\n);}
\foreach \n in {1,2}{\draw [line width=.6pt] (4, 4+.1+\n)-- (7, 5+.1+\n);}
\foreach \n in {1,2}{\draw [line width=.6pt] (8, 8+.1+\n)-- (11, 9+.1+\n);}
\foreach \n in {1,2}{\draw [line width=.6pt] (12, 12+.1+\n)-- (15, 13+.1+\n);}
\foreach \n in {1,2}{\draw [line width=.6pt] (0-.3, 0+.2+\n)-- (3-.3, 1+.2+\n);}
\foreach \n in {1,2}{\draw [line width=.6pt] (4-.3, 4+.2+\n)-- (7-.3, 5+.2+\n);}
\foreach \n in {1,2}{\draw [line width=.6pt] (8-.3, 8+.2+\n)-- (11-.3, 9+.2+\n);}
\foreach \n in {1,2}{\draw [line width=.6pt] (12-.3, 12+.2+\n)-- (15-.3, 13+.2+\n);}

\onetwomid{0}{1}
\onetwomid{4}{5}
\onetwomid{8}{9}
\onetwomid{12}{13}
\onetwomid{3}{2}
\onetwomid{7}{6}
\onetwomid{11}{10}
\onetwomid{15}{14}

\draw (3+.3, 1) node [anchor=north] {$v_2$};
\draw (7+.3, 4) node [anchor=north] {$\rho^3v_3$};
\draw (11+.3, 9) node [anchor=north] {$P^2v_2$};
\draw (15+.3, 11) node [anchor=north] {$\rho^{10}v_4$};

\end{tikzpicture}
\caption{ The $E_{\infty}$ page of $\mathfrak{M}(\bbQ)[h_1^{-1}]$ up
  to Milnor-Witt stem 15. The notational conventions of figure
  \ref{fig:Q-RhoBSSInfinity} apply here.}
\label{fig:Q-EInfinity}

\end{figure}
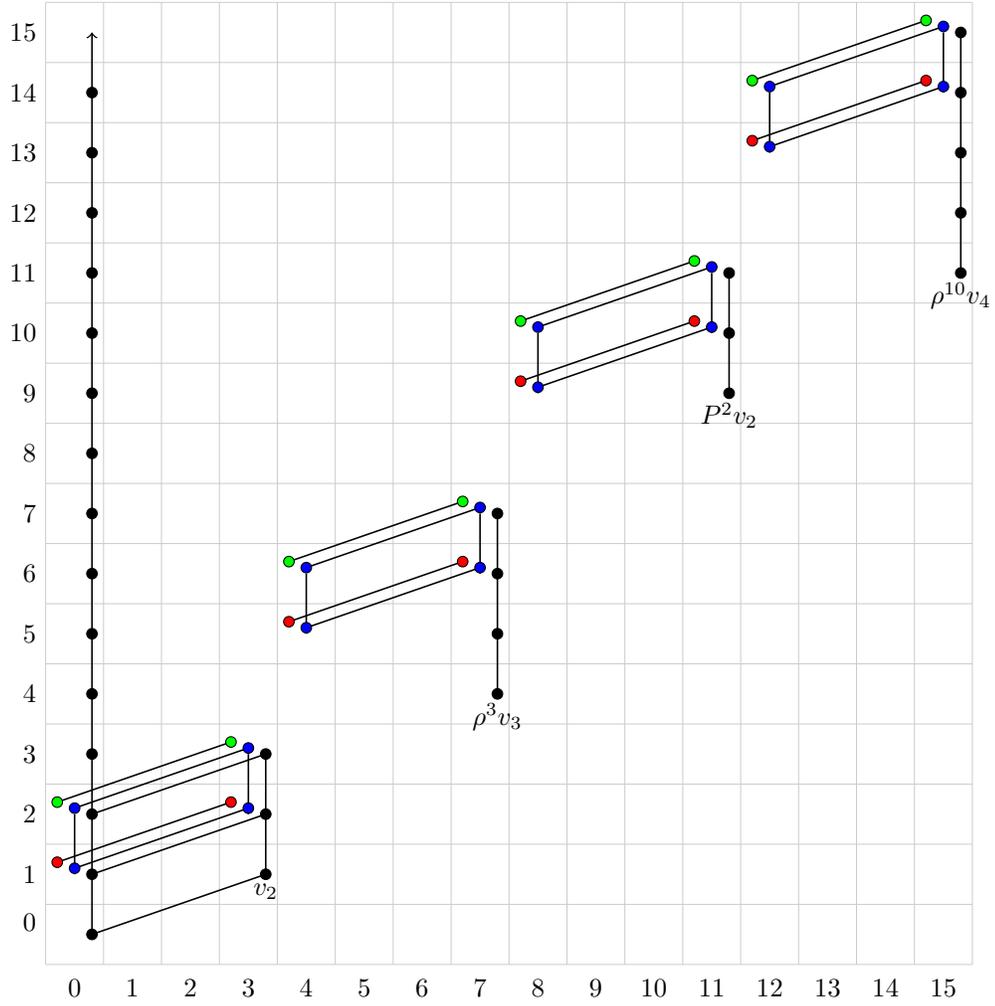

\end{document}